\documentclass{article}
\usepackage{amsmath,amssymb,amsthm}
\usepackage{mathdots}   

\numberwithin{equation}{section}
\newtheorem{thm}{Theorem}[section]
\newtheorem{lem}[thm]{Lemma}
\newtheorem{prop}[thm]{Proposition}
\newtheorem{cor}[thm]{Corollary}
\newtheorem{rem}[thm]{Remark}

\newcommand{\B}{{\Bbb B}}
\newcommand{\G}{{\Bbb G}}

\newcommand{\Z}{{\Bbb Z}}
\newcommand{\R}{{\Bbb R}}
\newcommand{\C}{{\Bbb C}}
\newcommand{\N}{{\Bbb N}}
 
\newcommand{\U}{{\Bbb U}}
\newcommand{\X}{{\Bbb X}}
\newcommand{\calC}{{\cal C}}
\newcommand{\calD}{{\cal D}}
\newcommand{\calH}{{\cal H}}
\newcommand{\calO}{{\cal O}}
\newcommand{\calR}{{\cal R}}
\newcommand{\calS}{{\cal S}}
\newcommand{\calU}{{\cal U}}
\newcommand{\frX}{{\frak X}}
\newcommand{\pair}[2]{\left\langle {#1}, \, {#2}\right\rangle}
\newcommand{\set}[2]{\left\{\left.#1\vphantom{#2}\:\right\vert\:#2\right\}}
\newcommand{\wt}{\widetilde}
\newcommand{\what}{\widehat}
\newcommand{\fra}{{\frak a}}
\newcommand{\frp}{{\frak p}}
\newcommand{\lam}{{\lambda}}
\newcommand{\Lam}{{\Lambda}}
\newcommand{\ve}{{\varepsilon}}
\newcommand{\alp}{{\alpha}}  
\newcommand{\vphi}{{\varphi}} 
\newcommand{\eps}{{\epsilon}} 
\newcommand{\slit}{\vspace{5mm}}
\newcommand{\mslit}{\vspace{3mm}}
\newcommand{\real}{{\rm Re}}
\newcommand{\Hom}{{\rm Hom}}
\newcommand{\ol}[1]{\overline{#1}}
\newcommand{\hec}{{\calH(G,K)}}
\newcommand{\CKX}{{\calC^\infty(K \backslash X)}}
\newcommand{\SKX}{{\calS(K \backslash X)}}
\newcommand{\ch}{{\mbox{ch}}}
\newcommand{\abs}[1]{\left\lvert{#1}\right\rvert} 
\newcommand{\bsqcup}{\mathop{\textstyle{\bigsqcup}}}
\newcommand{\bcup}{\mathop{\textstyle{\bigcup}}}
\newcommand{\dint}[1]{\displaystyle{\int_{#1}}}
\newcommand{\dprod}[1]{\displaystyle{\prod_{#1}}}
\newcommand{\dsqcup}[1]{\displaystyle{\bigsqcup_{#1}}}
\newcommand{\shita}[2]{\substack{#1\\#2}} 
\newcommand{\twomatrix}[4]{\begin{pmatrix}
                           {#1} & {#2}\\
                           {#3} & {#4}
                          \end{pmatrix}}
\newcommand{\twovector}[2]{\begin{pmatrix}{#1}\\{#2}\end{pmatrix}}
\newcommand{\twomatrixplus}[4]{\left(\begin{array}{c|c}
                           {#1}  & {#2}\\
                           \hline
                           {#3} &  {#4}
                          \end{array}\right)}
\newcommand{\twomatrixminus}[4]{\begin{array}{cc}
                           {#1}  & {#2}\\
                           {#3} &  {#4}
                          \end{array}}
\newcommand{\mapdownr}[1]{\Big\downarrow
   \rlap{$\vcenter{\hbox{$\scriptstyle#1$}}$ }}
\newcommand{\bfd}{\mathbf{d}}

\begin{document}


%
%


\title{Spherical functions on the space of \\$p$-adic unitary hermitian matrices}

\author{Yumiko Hironaka \\
\it Department of Mathematics\\
\it Faculty of Education and Integrated Sciences, Waseda University\\
\it Nishi-Waseda, Tokyo, 169-8050, JAPAN
\\
\it hironaka@waseda.jp
\\
\\
Yasushi Komori\\
\it Department of Mathematics\\
\it Faculty of Science, Rikkyo University\\
\it Nishi-Ikebukuro, Tokyo, 171-8501, JAPAN
\\
\it komori@rikkyo.ac.jp
}

\maketitle


\begin{abstract}
  We investigate the space $X$ of unitary hermitian matrices over
  $\frp$-adic fields through spherical functions. 
  First we consider Cartan decomposition of $X$, and give precise
  representatives for fields with odd residual characteristic, i.e.,
  $2\notin \frp$. In the latter half we assume odd residual
  characteristic, and give explicit formulas of typical spherical
  functions on $X$, where Hall-Littlewood symmetric polynomials of
  type $C_n$ appear as a main term, parametrization of all the
  spherical functions. By spherical Fourier transform, we show the
  Schwartz space $\SKX$ is a free Hecke algebra $\hec$-module of rank
  $2^n$, where $2n$ is the size of matrices in $X$, and give the
  explicit Plancherel formula on $\SKX$.
\end{abstract}


\medskip



\medskip

This research is partially supported by Grant-in-Aid for scientific Research (C): 22540045, 24540031, 25400026.



\setcounter{section}{-1}
\section{Introduction}

Let $\G$ be a reductive algebraic group and $\X$ a $\G$-homogeneous affine algebraic variety, where everything is assumed to be defined over a $\frp$-adic field $k$. We denote by $G$ and $X$ the sets of $k$-rational points of $\G$ and $\X$, respectively, take a maximal compact subgroup $K$ of $G$, and consider the Hecke algebra $\hec$. Then, a nonzero $K$-invariant function on $X$ is called {\it a spherical function on $X$} if it is an $\hec$-common eigenfunction.

Spherical functions on homogeneous spaces comprise an interesting topic to investigate and a basic tool to study harmonic analysis on $G$-space $X$. They have been studied as spherical vectors of distinguished models, Shalika functions and Whittaker-Shintani functions, there are close relation to the theory of automorphic forms, and spherical functions may appear in local factor of global object like Rankin-Selberg convolution and Eisenstein series. The theory of spherical functions has also an application of classical number theory, e.g. local densities of representations of quadratic forms or hermitian forms.  

To obtain explicit formulas of spherical functions is one of basic problems, and it has been done for the group case by I.~G.~Macdonald and afterwards by W.~Casselman by a representation theoretical method(\cite{Mac71}, \cite{Cas}). 
For homogeneous spaces, there are results 
mainly for the case that the space of spherical functions attached to each Satake parameter is of dimension one (e.g. \cite{CasS}, \cite{Alt}, \cite{KMS}, \cite{Offen}, \cite{Sake}). 

   
%
The first author has given a general expression of spherical functions on the basis of data of the group $G$ and functional equations of spherical functions when the dimension is not necessarily one, which is a development of a method of Casselman and Shalika (\cite{Cas}, \cite{CasS}), and a sufficient condition to obtain functional equations with respect to the Weyl group of $G$ (cf. \cite{JMSJ}, \cite{Hamb}, a refinement \cite{French}). 

In the present paper, we will use the above result to obtain a precise explicit formula of spherical functions of the unitary hermitian space $X$, which is a symmetric space of type $C_n$, and advance harmonic analysis on the space $X$. 
 This space $X$ is an important classical space from the view point of arithmetic of forms, and spherical functions on $X$ have a close relation to $p$-adic hermitian Siegel forms (cf. Appendix C and \cite{Oda}).
Here and henceforth we fix an unramified quadratic extension $k'$ of $k$, consider hermitian forms and unitary groups with respect to the extension $k'/k$, and denote by $A^* \in M_{n m}(k')$ the conjugate transpose of $A \in M_{m n}(k')$. We denote by $q$ the cardinality of the residue class field of $k$. We set
\begin{eqnarray*}
&&
G = U(j_{2n}) = \set{g \in GL_{2n}(k')}{g^*j_{2n}g = j_{2n}}, \quad j_{2n} = \begin{pmatrix}
0 & {} & 1\\ {} & \iddots
& {}\\1 & {} & 0
\end{pmatrix} \in GL_{2n}(k'),\\
&&
X = \set{x \in G}{x^* = x, \; \Phi_{xj_{2n}}(t)  = (t^2-1)^n }, 
\end{eqnarray*}
where $\Phi_y(t)$ is the characteristic polynomial of the matrix $y$, and $G$ acts on $X$ by
$$
g \cdot x = gxg^*, \quad (g \in G, \; x \in X).
$$

In \S \ref{sec:1}, we consider the Cartan decomposition of $X$, i.e. $K$-orbit decomposition of $X$, where $K = G \cap GL_{2n}(\calO_{k'})$, and show the following (cf. Theorem~\ref{thm: Cartan}, Proposition~\ref{Cartan n=1}, Theorem~\ref{prop: even}, Theorem~\ref{thm: G-orbits}).

\slit
\noindent
{\bf Theorem~1} (1) {\it If $k$ has odd residual characteristic, the $K$-orbit decomposition of $X$ is given by }
\begin{eqnarray} \label{intro-Cartan}
X = \bigsqcup_{\lam \in \Lam_n^+}\, K \cdot x_\lam,
\end{eqnarray}
{\it where}
\begin{eqnarray*} 
&&
x_\lam = Diag(\pi^{\lam_1}, \ldots, \pi^{\lam_n}, \pi ^{-\lam_n}, \ldots, \pi^{-\lam_1}) \in X, \\
&&
\Lam_n^+ = \set{\lam \in \Z^n}{\lam_1 \geq \lam_2 \geq \cdots \geq \lam_n \geq 0},
\end{eqnarray*}

\noindent
(2) {\it If $k$ has even residual characteristic, there exists a $K$-orbit which does not contain any diagonal matrix, for each $n \geq 1$. }

\noindent
(3) {\it There are precisely two $G$-orbits in $X$, independent of the residual characteristic of $k$.}

\slit
We introduce a typical spherical function $\omega(x; s)$ on $X$ as follows (for details, see \S2.1) :
\begin{eqnarray} \label{intro def sph}
\omega(x; s) = \int_{K}\, \prod_{i=1}^n \abs{d_i(k \cdot x)}^{s_i + \ve_i} dk, 
\end{eqnarray}
where $d_i(y)$ is the determinant of the lower right $i$ by $i$ block of $y$, \; $\ve \in \C^n$ is a certain fixed number, $dk$ is the Haar measure on $K$. 
The above integral is absolutely convergent if $\real(s_i) \geq -\real(\ve_i), \; 1 \leq i \leq n$, continued to a rational function of $q^{s_1}, \ldots, q^{s_n}$, and becomes a element of $\CKX$ for each $s \in \C^n$. 

It is convenient to introduce a new variable $z \in \C^n$ which is related to $s$ by
\begin{eqnarray} \label{intro new var}
s_i = -z_i + z_{i+1}, \; (1 \leq i \leq n), \quad s_n = -z_n,
\end{eqnarray}
and we write $\omega(x; z) = \omega(x; s)$. 
As for $z$-variable, the $\hec$-action on $\omega(x; z)$ can be written as (cf. (\ref{satake iso})) 
\begin{eqnarray}
\big( f * \omega(\; ; z)\big)(z) = \lam_z(f) \omega(x;z), \quad f \in \hec,
\end{eqnarray}
where $*$ is the convolution action of $\hec$ on $\CKX$, $\lam_z$ is the Satake transform $\hec \stackrel{\sim}{\rightarrow} \C[q^{\pm 2z_1}, \ldots, q^{\pm 2z_n}]^W$, and $W$ is the Weyl group of $G$.

Denote by $\Sigma_s^+$ (resp. $\Sigma_\ell^+$) the set of all positive short (resp. long) roots of $G$ (cf.~\eqref{short and long}).
Then we have the following (Theorem~\ref{th: feq}, Theorem~\ref{th: W-inv}):

\slit
\noindent
{\bf Theorem 2}
(1) {\it For every $\sigma \in W$, one has
$$
\omega(x; z) = \Gamma_{\sigma}(z) \cdot \omega(x; \sigma(z)),
$$
where
$$
\Gamma_\sigma(z) = \prod_{\alp \in \Sigma_s^+(\sigma)}\, \frac{1 - q^{\pair{\alp}{z}-1}}{q^{\pair{\alp}{z}}-q^{-1} }, \quad \Sigma_s^+(\sigma) = \set{\alp \in \Sigma_s^+}{-\sigma(\alp) \in\Sigma_s^+}.
$$
} 

\mslit
\noindent
(2) {\it The function $G(z) \cdot \omega(x; z)$ is contained in $\C[q^{\pm z_1}, \ldots, q^{\pm z_n}]^W (= \calR$, say$)$, where}
$$
G(z) = \prod_{\alp \in \Sigma_s^+}\, \frac{1 + q^{\pair{\alp}{z}}}{1 - q^{\pair{\alp}{z}-1}}.
$$

\slit
In \S \ref{sec:3}, we obtain the explicit formula for $\omega(x_\lam; z)$ for each $\lam \in \Lam_n^+$ (Theorem~3.1) by the method of \cite[Theorem~2.6]{French}.
Since $\omega(x; z)$ is $K$-invariant for $x$, it is enough to consider the explicit formula for $x_\lam, \; \lam \in \Lam_n^+$ by Theorem~1.

\slit
\noindent
{\bf Theorem 3}
{\it For each $\lam \in \Lam_n^+$, one has}
\begin{equation}
\omega(x_\lam; z)
 =
\frac{(1-q^{-2})^n}{w_{2n}(-q^{-1})} \cdot G(z)^{-1} \cdot c_\lam \cdot Q_\lam(z),
\end{equation}
{\it where $G(z)$ is given in Theorem 2, and} 
\begin{align*}
&
w_m(t) = \prod_{i=1}^{m}(1 - t^i), \quad 
c_\lam = (-1)^{\Sigma_i\, \lam_i(n-i+1)} q^{-\Sigma_i\,  \lam_i (n-i + \frac12)}, 
\\
&
Q_\lam(z) = 
\sum_{\sigma \in W}\, \sigma\left( q^{-\pair{\lam}{z}} c(z) \right), \quad c(z) = 
\prod_{\alp \in \Sigma_s^+}\, \frac{1 + q^{\pair{\alp}{z}-1}}{1 - q^{\pair{\alp}{z}}}
\prod_{\alp \in \Sigma_\ell^+}\, \frac{1 - q^{\pair{\alp}{z}-1}}{1 - q^{\pair{\alp}{z}}}.
\end{align*}

\bigskip
By Theorem~2, we see $Q_\lam(z)$ is a polynomial in $\calR$. On the other hand, this is a specialization of Macdonald polynomial of type $C_n$ up to scalar multiple. Here Macdonald polynomials were originally introduced by Macdonald as a unification of various families of orthogonal polynomials, such as Jack polynomials, Hall-Littlewood polynomials, Weyl characters.
It is known that the set $\set{Q_\lam(z)}{\lam \in \Lam_n^+}$ forms a $\C$-basis for $\calR$ and $Q_{\bf0}(z)$ is a constant(for details, see Remark~\ref{rem Mac} and Appendix \ref{sec:appB}). 
%
We modify $\omega(x;z)$ by
\begin{eqnarray} \label{modify-omega}
&&
\Psi(x; z) = \omega(x; z) \big{/} \omega(1_{2n}; z), 
\end{eqnarray}
which is an element of $\calR$, and
\begin{eqnarray*}
&&
\omega(1_{2n}; z) = 
\frac{(1-q^{-1})^n w_n(-q^{-1})^2}{w_{2n}(-q^{-1})} \times G(z)^{-1},\\
&&
\Psi(x_\lam; z) = \frac{(1+q^{-1})^n}{w_n(q^{-1})^2} \cdot c_\lam \cdot Q_\lam(z), \qquad (\lam \in \Lam_n^+).
\end{eqnarray*}
Employing $\Psi(x; z)$ as Kernel function, we consider the spherical Fourier transform on the Schwartz space $\SKX$ in \S 4:
\begin{equation}
\begin{array}{lccl}
F : & \SKX & \longrightarrow & \calR \nonumber\\
{} & \vphi & \longmapsto & F(\vphi) = \int_{X} \vphi(x) \Psi(x; z)dx,
\end{array}
\end{equation}
where $dx$ is a $G$-invariant measure on $X$.
We will show the following (cf.~Theorem~\ref{th: sph trans}, Corollary~\ref{cor:4.2}, Theorem~\ref{th: Plancherel}):

\slit
\noindent
{\bf Theorem 4}
(1) {\it The spherical Fourier transform $F$ is an $\hec$-module isomorphism, in particular, $\SKX$ is a free $\hec$-module of rank $2^n$.}

\medskip
\noindent
(2) {\it For each $z \in \C^n$, the set $\set{\Psi(x; z+u)}{u \in \{0, \frac{\pi\sqrt{-1}}{\log q}\}^n}$ forms a basis for the spherical functions on $X$ corresponding to $\lam_z$, where $\Psi(x; z)$ is given by (\ref{modify-omega}).}

\medskip
\noindent
(3) ({\it Plancherel formula}) {\it Set a measure $d\mu(z)$ on $\fra^*$ by}
$$
d\mu = \frac{1}{n!2^n}\cdot \frac{w_n(-q^{-1})^2}{(1+q^{-1})^n} \cdot \frac{1}{\abs{c(z)}^2} dz,\qquad \fra^* = \left\{\sqrt{-1}\left(\R \big{/} \frac{2\pi}{\log q} \Z \right) \right\}^n,
$$ 
{\it where $dz$ is the Haar measure on $\fra^*$. By an explicitly given normalization of $dx$ on $X$ one has} 
$$
\int_{X} \vphi(x) \ol{\psi(x)} dx = \int_{\fra^*}\, F(\vphi)(z) \ol{F(\psi)(z)} d\mu(z), \quad (\vphi, \; \psi \in \SKX).
$$

\bigskip
In \cite{Oda}, we have investigated spherical functions on a similar space $X_T$ associated with each nondegenerate hermitian matrix $T$, and obtained functional equations of hermitian Siegel series as an application. Both spaces, $X_T$ and the present $X$, are isomorphic to $U(2n)/U(n)\times U(n)$ over the algebraic closure of $k$, and the former realization was useful for the application to hermitian Siegel series. But it was not easily understandable, and we could not obtain its Cartan decomposition, nor complete parametrization of spherical functions. 
We discuss the correspondence between both spaces in Appendix \ref{sec:appC}, and see many results on the present space $X$ are inherited to the former spaces $X_T$. Further applications for the theory of automorphic forms will be expected.

\bigskip
Throughout of this article except Appendix \ref{sec:appA}, where we explain about unitary hermitian matrices in a general setting, we denote by $k$ a non-archimedian local field of characteristic $0$, fix an unramified quadratic extension $k'$ and consider unitary and hermitian matrices with respect to $k'/k$.
We fix a prime element $\pi$ of $k$, denote by $v_\pi(\; )$ the additive value on $k$, and normalize the absolute value $\abs{\; }$ on $k^\times$ by $\abs{\pi}^{-1} = q = \sharp(\calO_k/(\pi))$. 
We also fix a unit $\eps \in \calO_k^\times$ for which $k' = k(\sqrt{\eps})$. We may take $\eps$ such as $\eps - 1 \in 4\calO_k^\times$, so that $\{1, \, \frac{1+\sqrt{\eps}}{2} \}$ forms an $\calO_k$-basis for $\calO_{k'}$ (cf.~\cite[64.3 and 64.4]{Ome}). From \S \ref{sec:2} to \S \ref{sec:4}, we assume that $q$ is odd.


\bigskip
Acknowledgment: The authors would like to express their thanks to the reviewers for their careful reading of the manuscript.  
 
\vspace{2cm}

\section{The space $X$ and its $K$-orbit decomposition and $G$-orbit decomposition}
\label{sec:1}

Let $k'$ be an unramified quadratic extension of a $\frp$-adic field $k$ and consider hermitian matrices and unitary matrices with respect to $k'/k$. For a matrix $A \in M_{mn}(k')$, we denote by $A^* \in M_{nm}(k')$ its conjugate transpose with respect to $k'/k$, and say $A$ is {\it hermitian} if $A^* = A$.

We consider the unitary group
\begin{equation*}
G = G_n = \set{g \in GL_{2n}(k')}{g^*j_{2n}g = j_{2n}}, \qquad j_{2n} = \begin{pmatrix}
0 & {} & 1\\ {} & 
\iddots
& {}\\1 & {} & 0
\end{pmatrix} \in M_{2n},
\end{equation*}
the space $X$ of unitary hermitian matrices in $G$
\begin{equation} \label{space X}
X = X_n  = \set{x \in G}{x^* = x, \; \Phi_{xj_{2n}}(t) = (t^2-1)^n},
\end{equation}
and a supplementary space  $\wt{X}$ containing $X$ 
\begin{equation*}
\wt{X} = \wt{X}_n = \set{x \in G}{x = x^*},
\end{equation*}
where $\Phi_y(t)$ is the characteristic polynomial of the matrix $y$. 
It should be noted that \eqref{space X} implies $\det x=1$.
The group $G$ acts on $X$ and $\wt{X}$ by
$$
g \cdot x = gxg^* = x[g^*] = gxj_{2n}g^{-1}j_{2n}, \quad g \in G, \; x \in \wt{X}.
$$
As is explained in Appendix \ref{sec:appA}, we may understand $X$ as the set of $k$-rational points of a $G(\ol{k})$-homogeneous algebraic set $X(\ol{k})$ defined over $k$, where $\ol{k}$ is the algebraic closure of $k$. 

We fix a compact subgroup $K$ of $G$ by 
\begin{equation*} \label{def of K}
K = K_n = G \cap M_{2n}(\calO_{k'}),
\end{equation*}
which is maximal compact (cf.~\cite[\S9]{Satake}).
The main purpose of this section is to give the Cartan decomposition of $X$, i.e., the $K$-orbit decomposition of $X$ for odd $q$ (Theorem~\ref{thm: Cartan}), and $G$-orbit decomposition of $X$ (Theorem~\ref{thm: G-orbits}).

To start with, we recall the case of unramified hermitian matrices. The group $G_0 = GL_n(k')$ acts on the space $\calH_n(k') = \set{y \in G_0}{y^* = y}$ by $g \cdot y = gyg^*$, and there are two $G_0$-orbits in $\calH_n(k')$ determined by the parity of $v_\pi(\det(y))$. Setting $K_0 = GL_n(\calO_{k'})$, the Cartan decomposition is known (cf.~\cite{Jac}) as follows:
\begin{equation} \label{herm}
\calH_n(k') = \bigsqcup_{\lam \in \Lam_n}\, K_0 \cdot \pi^\lam,
\end{equation}
where
$$
\pi^\lam = Diag(\pi^{\lam_1}, \ldots, \pi^{\lam_n}), \quad
\Lam_n = \set{\lam \in \Z^n}{\lam_1 \geq \lam_2 \geq \cdots \geq \lam_n}. 
$$

\medskip
\begin{thm} \label{thm: Cartan}
Assume that $k$ has odd residual characteristic. Then, the $K$-orbit decomposition of $X_n$ is given as follows:
\begin{equation*} \label{K-orbits}
X_n = \bigsqcup_{\lam \in \Lam_n^+}\, K \cdot x_\lam,
\end{equation*}
where 
\begin{eqnarray*}
&&
\Lam_n^+ = \set{\lam \in \Z^n}{\lam_1 \geq \cdots \geq \lam_n \geq 0},\\
&&
x_\lam 
= Diag(\pi^{\lam_1}, \ldots, \pi^{\lam_n}, \pi^{-\lam_n}, \cdots, \pi^{-\lam_1}).
\end{eqnarray*}
\end{thm}
 
\medskip
For $a = (a_{ij}) \in M_{2n}(k')$, we set
\begin{eqnarray*}
-\ell(a) = \min\set{v_\pi(a_{ij})}{1 \leq i, j \leq 2n},
\end{eqnarray*}
and say an entry of $a$ to be {\it minimal} if its $v_\pi$-value is $-\ell(a)$.
For $g \in G$, we see $\ell(g) \geq 0$, since $v_\pi(\det(g)) = 0$.
For regularization of $x \in \wt{X}$, we often use elements in $K$ of the following type:
\begin{eqnarray}
&&
\twomatrix{h}{0}{0}{j_{n}h^{* -1}j_{n}}, \quad \mbox{for } h \in K_0, \nonumber \\
&& \label{K-elements}
\twomatrix{1_n}{aj_n}{0}{1_n}, \quad \mbox{for } a \in M_n(\calO_{k'}), \; a + a^* = 0.
\end{eqnarray}

\medskip
\begin{prop} \label{Cartan n=1}
Let $n = 1$. Then 
\begin{eqnarray*}
&& 
X_1 = \displaystyle{\bigsqcup_{\ell \geq 0}}\, K_1 \cdot \twomatrix{\pi^\ell}{0}{0}{\pi^{-\ell}} \sqcup \displaystyle{\bigsqcup_{1 \leq r \leq v_\pi(2)}}\, K_1 \cdot \twomatrix{\pi^{-r}(1-\eps)}{-\sqrt{\eps}}{\sqrt{\eps}}{\pi^r},
%
\end{eqnarray*}
where the latter union is empty if $q$ is odd.
\end{prop}

\begin{proof}
For $x = \twomatrix{a}{\beta}{\beta^*}{c} \in \wt{X}_1$, we have
$$
j_2 = j_2[x] = \twomatrix{a(\beta + {\beta}^*)}{ac + \beta^2}{ac + {\beta}^{*2}}{c(\beta+{\beta}^*)}.
$$ 
If $a = 0$ or $c = 0$, we have $\beta = \pm 1$, hence $a = c = 0$ and $x = \pm j_2 \notin X_1$.
Hence $ac \ne 0$ for $x \in X_1$, and we may assume
$$
x = \twomatrix{a}{b\sqrt{\eps}}{-b\sqrt{\eps}}{c}, \quad \begin{array}{l}
a,b,c \in k, \; ac + b^2\eps = 1,\\
v_\pi(a) \geq v_\pi(c), \; c \mbox{\, is a power of $\pi$}.
\end{array}
$$
If $v_\pi(c) = -\ell \leq 0$, then $v_\pi(b) \geq -\ell$, and  
$$
K \cdot x \ni \twomatrix{1}{-b\pi^\ell \sqrt{\eps}}{0}{1} \cdot \twomatrix{a}{b\sqrt{\eps}}{-b\sqrt{\eps}}{\pi^{-\ell}} = \twomatrix{\pi^\ell}{0}{0}{\pi^{-\ell}} \in X_1,
$$
where each $\ell$ gives different $K_1$-orbit. 

Next assume $v_\pi(c) = r > 0$, then $b \in \calO_k^\times$, $b^2 \eps \equiv 1\pmod{(\pi^2)}$ and 
$x \in K$.
Thus $q$ is even, $b \equiv 1 \pmod{(\pi)}$ and $x \equiv j_2 \pmod{(\pi)}$, since $\eps \in 1 + 4\calO_k^\times$. 
We may rewrite
$$
x = \twomatrix{\pi^r a'}{1+\pi^m\gamma}{1+\pi^m\gamma^*}{\pi^r}, \quad \begin{array}{l}
a' \in \calO_k, \; 0 \leq m \leq r, \\
\gamma \in \calO_{k'}, \; \gamma \in \calO_{k'}^\times \mbox{ if } m<r.
\end{array}
$$
Since $\Phi_{xj_2}(t) = t^2 - 1$, we have
$$
\pi^m(\gamma + \gamma^*) + 2 = 0, \quad \pi^{2m}N(\gamma) = \pi^{2r}a'.
$$
By the latter equation we have $m = r$, and setting $\gamma = b_0 + b_1 \frac{1+\sqrt{\eps}}{2}$, we have 
$$
1 + \pi^r\gamma = -(1 + \pi^r b_0)\sqrt{\eps}.
$$
Thus, we obtain
$$
K \cdot x \ni \twomatrix{1}{b_0 \sqrt{\eps}}{0}{1} \cdot x = 
\twomatrix{\pi^{-r}(1-\eps)}{-\sqrt{\eps}}{\sqrt{\eps}}{\pi^r} (= x_r, \mbox{say}), \; 1 \leq r \leq v_\pi(2).
$$
If $K \cdot x_r$ contained a diagonal matrix it must be $1_2$, and $kx_r = j_2kj_2$ for some $k \in K$. By the latter equation we get $\det(k) \equiv 0 \pmod{(\pi)}$, which is impossible for $k \in K$.
Similarly we may prove $x_r \notin K \cdot x_s$ if $r \ne s$, which completes the proof.
\end{proof}

\begin{lem}   \label{lem 2-1}
Let $n \geq 2$ and assume that $x \in \wt{X}_n$ has a minimal entry in the diagonal. Then $K \cdot x$ contains a hermitian matrix of type
$$
\left( \begin{array}{c|c|c}
\pi^\ell & 0 & 0\\ 
\hline
0 & y & 0\\  
\hline
0 & 0 & \pi^{-\ell}
\end{array} \right), \qquad y \in \wt{X}_{n-1} \cap M_{2n-2}(\pi^{-\ell}\calO_{k'}), 
$$
where $\ell = \ell(x)$. If $x \in X_n$, then the above $y \in X_{n-1}$.
\end{lem}

\begin{proof}
By the action of $W$, we may assume the $(2n, 2n)$-entry is minimal.
Then, by (\ref{herm}) and (\ref{K-elements}), we see there is some $x' \in K\cdot x$ whose lower right $n$ by $n$ block has the form
$$
\left( \begin{array}{c|c}
{} & 0\\
*  & \vdots\\
{} & 0\\
\hline
0 \cdots 0 & \pi^{-\ell} \end{array} \right). 
$$
Then, by taking a suitable matrix of type 
$$
h = \twomatrix{1_n}{A}{0}{1_n} \in K, \quad 
A = \twomatrixplus {-{a_n}^*\, \cdots\,  {-a_2}^*}{a}{0}{\begin{array}{c}a_2\\\vdots\\a_n\end{array}} \in M_n(k'), \quad a + {a}^* = 0,
$$
$h\cdot x'$ becomes the following form
\begin{equation} \label{lem1 mid}
\left( \begin{array}{c|ccc|c}
{c} & {c_2} & {\cdots} & {c_{2n-1}} & b\\
\hline
{{c_2}^*} & {} & {} & {} & 0\\
{\vdots} & {} & {y} & {} & \vdots\\
{{c_{2n-1}}^*} & {} & {} & {} & 0\\
\hline
b^* & 0 & \cdots & 0 & \pi^{-\ell} \end{array} \right) \in \wt{X}_n, \quad 
\begin{array}{l}
b = \pi^{-\ell}(b_0 + b_1\frac{1+\sqrt{\eps}}{2}), \; \mbox{with }\\
b_0, b_1 \in \calO_k, \; b_1 = 0 \mbox{ or }b_1 \notin 2\calO_k,\\
c \in \pi^{-\ell} \calO_k, \; c_i \in \pi^{-\ell}\calO_{k'}.
\end{array}
\end{equation}
Since $j_{2n}[h\cdot x'] = j_{2n}$, we have
$$
b = 0, \; c = \pi^\ell, \; c_i = 0 \; (2 \leq i \leq 2n-1),
$$
and then $y \in \wt{X}_{n-1}$, and it is clear that $y \in X_{n-1}$ if $x \in X_{n}$, which completes the proof. 
\end{proof}

\begin{lem}   \label{lem 2-2}
Let $n \geq 2$ and assume that $x \in \wt{X}_n$ has a minimal entry outside of the diagonal and the anti-diagonal. Then $K \cdot x$ contains a hermitian matrix of type
$$
\left( \begin{array}{c|c|c}
\twomatrixminus{\pi^\ell}{0}{0}{\pi^\ell} & 0 & 0\\
\hline 
0 & y \vphantom{A^{A^n}_{A_h}\;}& 0\\
\hline  
0 & 0 & \twomatrixminus{\pi^{-\ell}}{0}{0}{\pi^{-\ell}}
\end{array}\right), \qquad y \in \wt{X}_{n-2} \cap M_{2n-4}(\pi^{-\ell}\calO_{k'}), 
$$
where $\ell = \ell(x)$. If $x \in X_n$, the the above $y \in X_{n-2}$.
\end{lem}

\begin{proof}
By the assumption, the minimal entries appear in pair not in the anti-diagonal. Then, by the action of $W$, we may assume the $(2n,2n-1)$-entry and the $(2n-1, 2n)$-entry are minimal. 
Then, by (\ref{herm}) and (\ref{K-elements}), we see there is some $x' \in K\cdot x$ whose lower right $n$ by $n$ block has the form
$$
\left( \begin{array}{c|cc}
{} & 0 & 0\\
*  & \vdots &\vdots \\
{} & 0 & 0\\
\hline
0 \cdots 0 & {\pi^{-\ell}} & {0}\\
0 \cdots 0 & 0 & {\pi^{-\ell}} \end{array} \right). 
$$
Taking the similar procedure (twice) to the proof of Lemma~\ref{lem 2-1}, we obtain a matrix of the required form.
\end{proof}

\begin{lem}   \label{lem 2-3}
Let $x \in \wt{X}_n$ with $n \geq 2$. Assume that any minimal entry of $x \in \wt{X}_n$ stands in the anti-diagonal but not all the entries of the anti-diagonal are minimal. Then $K \cdot x$ contains a hermitian matrix of the same type as in Lemma~\ref{lem 2-2}.
\end{lem}

\begin{proof}
Set $\ell = \ell(x)$. By the action of $W$, we may assume 
$$
x = \twomatrixplus
{*}{\twomatrixminus{a}{\xi}{c}{b}} {*} {*}, \qquad v_\pi(\xi) = -\ell, \quad a,b,c \in \pi^{-\ell+1}\calO_{k'}. 
$$
Then, for the matrix
$$
h = \left( \begin{array}{c|c|c}
  {\twomatrixminus{1}{1}{0}{1}} & {0} & {0}\\
  \hline
{0} & 1_{2n-4} & {0}\\
\hline
{0} & {0} & {\twomatrixminus{1}{-1}{0}{1}}
\end{array} \right) \in K,
$$
$\ell(h\cdot x) = \ell(x)$ and the $(1, 2n-1)$-entry of $h\cdot x$ is equal to $-\xi + a-b+c$ and minimal.
Then, $h \cdot x$ satisfies the assumption of Lemma~\ref{lem 2-2}, and the result follows from this.
\end{proof}

\begin{lem}   \label{lem 2-4}
Let $x \in \wt{X}_n$ with $n \geq 2$. 
Assume that any minimal entry of $x$ stands in the anti-diagonal and that all the entries in the anti-diagonal are minimal. Denote by $\xi_i$ the $(i, 2n+1-i)$-entry of $x, \; (1 \leq i \leq n)$. Then

{\rm (i)} One has $\ell(x) = 0$, $x \in K$, and $\xi_i \equiv \pm 1 \pmod{(\pi)}, \; 1 \leq i \leq n$.

{\rm (ii)} If $\xi_i \not\equiv \xi_j \pmod{(\pi)}$ for some $i$ and $j$, which occurs only when $2 \notin (\pi)$, then $K\cdot x$ contains a hermitian matrix of type 
$$
\left( \begin{array}{l|c|l}
  {1_2} & {0} & {0}\\[2mm]
  \hline
  {0} & y & {0}\\[2mm]
  \hline
  {0} & {0} & 1_2
\end{array} \right) \in K, \quad y \in \wt{X}_{n-2} \cap M_{2n-4}(\calO_{k'}),
$$
where we understand the above matrix is $1_4$ when $n = 2$. If $x \in X_n$, then the above $y \in X_{n-2}$.

{\rm (iii)} If $k$ has odd residual characteristic and $x \equiv \pm j_{2n}\pmod{(\pi)}$, then $x \notin X_n$.  
\end{lem}

\begin{proof}
(i) Set $\ell = \ell(x)$. By the assumption, we see $v_\pi(\det x) = 2v_\pi(\xi_1\cdots \xi_n) = -2\ell n$, which must be $0$, hence $\ell = 0$ and
$$
x \equiv 
\begin{pmatrix}
{} & {} & {} & {} &                   {} & {\xi_1}\\
         {} & {0} & {} & {} &      \iddots 
 & {}\\
           {}  &  {}  &          {}  &  {\xi_n} & {} & {}\\
 {} & {} &              {{\xi_n}^*} & {} & {} & {}\\
      {}   &  \iddots 
& {}           & {} & {0} & {}\\
{{\xi_1}^*} & {}    & {}           & {} & {} & {}
\end{pmatrix} \pmod{(\pi)}.
$$
Since $j_{2n}[x] = j_{2n}$, we have $\xi_i^2 \equiv {\xi_i}^{* 2} \equiv 1 \pmod{(\pi)}$, hence $\xi_i \equiv \pm 1 \mod{(\pi)}$ for every $i$.

(ii) Now we assume $2 \notin (\pi)$ and $\xi_i \not\equiv \xi_j \pmod{(\pi)}$ for some $i$ and $j$.
We may assume $\xi_1 \not\equiv \xi_2 \pmod{(\pi)}$ by the action of $W$, and write 
$$
x = \left( \begin{array}{c|c|c}
  {\twomatrixminus{a}{b}{{b}^*}{d}} & {*} & \twomatrixminus{c}{\xi_1}{\xi_2}{f}\\[2mm]
  \hline
{*} & {*} & {*}\\[2mm]
\hline
\twomatrixminus{{c}^*}{{\xi_2}^*}{{\xi_1}^*}{{f}^*} & {*} & \twomatrixminus{g}{h}{{h}^*}{r}
\end{array} \right).
$$
For $h \in K$ with 
$$
h = \left( \begin{array}{c|c|c}
  {\twomatrixminus{1}{0}{0}{1}} & {0} & {0}\\[2mm]
  \hline
{0} & 1_{2n-4} & {0}\\[2mm]
\hline
{\twomatrixminus{1}{0}{0}{-1}} & {0} & {\twomatrixminus{1}{0}{0}{1}}
\end{array} \right) \in K,
$$
the lower right $2$ by $2$ block of $h\cdot x$ becomes 
$$
\twomatrix{a+c+{c}^*+g}{-b+h+\xi_1-{\xi_2}^*}{-{b}^*+{h}^*+{\xi_1}^* - \xi_2}{d-f-{f}^*+r} 
\equiv \twomatrix{0}{2\xi_1}{2\xi_1}{0} \pmod{(\pi)},
$$
which is unimodular hermitian of size $2$, since $2 \notin (\pi)$, and we may change it into $1_2$. Then, the similar procedure (twice) to the proof of Lemma~\ref{lem 2-1}, we see $K\cdot x$ contains an element of the required form.

(iii) If $q$ is odd and $x \equiv cj_{2n} \pmod{(\pi)}$ with $c = \pm 1$, then
$\Phi_{xj_{2n}}(t) \equiv (t-c)^{2n} {\not\equiv} (t^2-1)^n \pmod{(\pi)}$, hence $x \notin X_n$.
\end{proof}

\medskip
Now Theorem~\ref{thm: Cartan} follows from Lemmas \ref{lem 2-1} to \ref{lem 2-4}.

We consider even residual case.
 
\medskip
\begin{lem}  \label{lem even case}
Assume that $q$ is even and $x \in X_n \cap K$ satisfies $x \equiv j_{2n} \pmod{(\pi)}$. 
Then, $K \cdot x$ does not contain any diagonal matrix, and represented by a matrix of the following type:
\begin{eqnarray*}
E_n(\mu) =   \begin{pmatrix}
    \pi^{-\mu_n}(1-\eps) &&&&& -\sqrt{\eps} \\
    & \ddots &&& \iddots & \\
    &&\pi^{-\mu_1}(1-\eps)& -\sqrt{\eps}&& \\
    &&\sqrt{\eps}&\pi^{\mu_1} && \\
    & \iddots &&& \ddots & \\
    \sqrt{\eps} &&&&& \pi^{\mu_n}
  \end{pmatrix},
\end{eqnarray*}
where each empty place means zero-entry, and
$$
\mu \in \Lam_n^{(2)} = \set{\mu \in \Lam_n^+}{v_\pi(2) \geq \mu_1 \geq \cdots \geq \mu_n \geq 1}.
$$
\end{lem}

\begin{proof}
If $K \cdot x$ contains a diagonal matrix, it must contain $1_{2n}$, since $x \in K$.
Assume $k \cdot x = 1_{2n}$ for some $k \in K$, and write 
$$
k = \twomatrix{a}{b}{c}{d}, \quad a,b,c,d \in M_n(\calO_{k'}).
$$
Then, since $k$ satisfies $kx = j_{2n}k j_{2n}$, we have
\begin{eqnarray*}
k \equiv \twomatrix{a}{b}{j_n a}{j_n b} \pmod{(\pi)},
\end{eqnarray*}
which contradicts to the fact $\det(k) \in \calO_{k'}^\times$, hence $K \cdot x$ does not contain any diagonal matrix.

Next we will show that $K \cdot x$ contains an $E_n(\mu)$ of the above type.
Since $j_{2n} \notin X_n$, we may write 
\begin{equation*} \label{even-1}
x = j_{2n} + \pi^m y, \quad m>0, \;  0 \neq y = y^* \in M_{2n}(\calO_{k'}), \; \ell(y) = 0,
\end{equation*}
where $\ell(y) = 0$ means that minimal entry of $y$ is a unit. If any entry in the anti-diagonal of $y$ is a unit and all the other entries contained in $(\pi)$, then $\det(y) \in \calO_k^\times$.
On the other hand 
$$
(t^2 - 1)^n = \Phi_{xj_{2n}}(t) = \det(t1_{2n}-(j_{2n}+\pi^m y)j_{2n}) = \det((t-1)1_{2m}-\pi^myj_{2n}),
$$
and
$$
0 = \Phi_{xj_{2n}}(1) = \det(-\pi^m y j_{2n}) = (-1)^n\pi^{2mn} \det(y),
$$
which is a contradiction.
Hence there is a minimal entry of $y$ not in the anti-diagonal. Then following the proof of Lemmas~\ref{lem 2-1} to \ref{lem 2-3},
there exists some $k \in K$, for which $k \cdot y$ becomes (\ref{lem1 mid}) with $\ell = 0$.
Since $k\cdot x = j_{2n} + \pi^m (k \cdot y)$ satisfies $(k\cdot x) \cdot j = j$, looking at the $2n$-th row, we have
\begin{eqnarray*}
c_i = 0, i \geq 2, \quad 2+2\pi^m b_0 + \pi^m b_1 = 0, 
\end{eqnarray*}
and
$$
k\cdot x = 
\left( \begin{array}{c|c|c}
\pi^mc & 0 & -(1+\pi^m b_0)\sqrt{\eps}\\
\hline
0 & \xi & 0\\
\hline
(1+\pi^m b_0) &   0 & \pi^m
\end{array}\right), \quad \xi \equiv j_{2(n-1)} \pmod{(\pi^m)},
$$
Then, by acting 
$$
\left( \begin{array}{c|c|c}
1 & 0 & b_0\sqrt{\eps}\\
\hline
0 & 1_{2(n-1)} & 0\\
\hline
0 & 0 & 1
\end{array}\right) \in K,
$$
$k \cdot x$ becomes 
$$
\left( \begin{array}{c|c|c}
\pi^m(1 - \eps) & 0 & -\sqrt{\eps}\\
\hline
0 & \xi & 0\\
\hline
\sqrt{\eps} & 0 & \pi^m
\end{array}\right), \quad \xi \equiv j_{2(n-1)} \pmod{(\pi^m)}.
$$
Repeating the same procedure, we conclude the proof.
\end{proof}

\bigskip
Now, by Lemmas~\ref{lem 2-1} to \ref{lem 2-3}, Lemma~\ref{lem 2-4}-(i), and Lemma~\ref{lem even case}, we see the following.

\begin{thm} \label{prop: even}
Assume that $q$ is even. Then 
\begin{eqnarray*}
X_n = \bcup_{r = 0}^n\, \bcup _{\shita{\lam \in \Lam_r^+}{\mu \in \Lam^{(2)}_{n-r}}}\     K \cdot x_{\lam, \mu}, \qquad x_{\lam,\mu} =   
\begin{pmatrix}
    D_r(\lambda) &0&0\\
    0& E_{n-r}(\mu) &0\\
    0&0& D_r(-\lambda)
  \end{pmatrix},
\end{eqnarray*}
where $D_r(\lambda)$ and $D_r(-\lambda)$ are related to $x_\lambda$ by
\begin{eqnarray*}
&&
x_\lam = \twomatrix{D_r(\lam)}{0}{0}{D_r(-\lam)}  \in X_r, 
\end{eqnarray*}
and $x_{\lam, \mu}$ is understood as $x_\lam$ (resp.~$E_n(\mu)$) if $r = n$ (resp.~$r = 0$). 
Further one has
$$
\bcup_{\lam \in \Lam_n^+}\, K \cdot x_\lam = \bsqcup_{\lam \in \Lam_n^+}\, K \cdot x_\lam \not\ni E_n(\mu) \qquad (\mu\in\Lam^{(2)}_{n}).
$$
\end{thm}

\bigskip
As a corollary of Theorem~\ref{thm: Cartan} and Theorem~\ref{prop: even}, we have the following. For $\lam \in \Lam_n^+$, we set $\abs{\lam} = \sum_{i=1}^n\, \lam_i$ and call $\lam$ to be {\it even} or {\it odd} according to the parity of $\abs{\lam}$.

\begin{thm} \label{thm: G-orbits}
There are precisely two $G$-orbits in $X_n$:
\begin{equation*} \label{G-orbits}
X_n = G \cdot x_0 \, \sqcup\,  G \cdot x_1,\quad x_0 = 1_{2n}, \; x_1 = Diag(\pi, 1, \ldots, 1, \pi^{-1}).
\end{equation*}
If $q$ is odd, then 
$$
G \cdot x_0 = \bigsqcup_{\shita{\lam \in \Lam_n^+}{even}}\, K \cdot x_\lam, \quad 
G \cdot x_1 = \bigsqcup_{\shita{\lam \in \Lam_n^+}{odd}}\, K \cdot x_\lam .
$$
If $q$ is even, $x_{\lam, \mu}$ is $G$-equivalent to $x_0$ if and only if $\abs{\lam} + \abs{\mu}$ is even.
\end{thm}

\begin{proof}
For unramified hermitian matrices, it is known that $\calH_m(k')$ has two $GL_m(k')$-orbits determined by the parity of $v_\pi(\det(y))$ and $H^1(\Gamma, U(y)(\ol{k})) \cong C_2$, where $\Gamma = Gal(\ol{k}/k)$, $y \in \calH_m(k')$ and $m \geq 1$. If $q$ is odd, each representative $x_\lam$ of $K$-orbit in Theorem~\ref{thm: Cartan} is diagonal; if $q$ is even, by the action of $B$, $x_{\lam, \mu}$ in Theorem~\ref{prop: even} becomes diagonal, hence there are at most two $G$-orbits in $X_n$, independent of the parity of $q$. 
We recall $G(\ol{k}), \; X_n(\ol{k})$ and $\star$-action in 
\eqref{star-action} in Appendix \ref{sec:appA} for $m = 2n$, and set 
$$
H(\ol{k}) = \set{h \in G(\ol{k})}{h \star 1_{2n} = 1_{2n}},
$$
then it is easy to see
\begin{align}  
H(\ol{k}) &= \set{\twomatrix{a}{b}{jbj}{jaj} \in GL_{2n}(\ol{k})}{a+bj, \; a-bj \in U(1_n)(\ol{k})} \quad (j = j_n) \notag \\
\label{stabilizer at 1_2n}
& \cong U(1_n)(\ol{k}) \times U(1_n)(\ol{k}). 
\end{align}
By the exact sequence of $\Gamma$-sets
$$
\begin{array}{lcccl}
1 \longrightarrow H(\ol{k}) \longrightarrow & G(\ol{k}) &\longrightarrow &X_n(\ol{k}) &\longrightarrow 1,
\\
{} & g & \longmapsto & g \star 1_{2n} & {}
\end{array}
$$
we have an exact sequence of pointed sets (cf.~\cite[I-\S 5.4]{Serre})
$$
1 \longrightarrow G \cdot 1_{2n} \longrightarrow X_n \longrightarrow H^1(\Gamma, H(\ol{k})) \stackrel{\eta}{\longrightarrow} H^1(\Gamma, G(\ol{k})).
$$
Since $\eta$ is a map from $C_2 \times C_2$ to $C_2$, ${\rm Ker}(\eta)$ cannot be trivial and $G \cdot 1_{2n} \ne X_n$. Hence there are at least two $G$-orbits in $X_n$, thus exactly two $G$-orbits and they are given as above.
\end{proof}

\vspace{2cm}

\section{Spherical function $\omega(x; s)$ on $X$}
\label{sec:2}

{\bf 2.1.} 
For simplicity, we write $j = j_{n}$, and take a Borel subgroup $B$ of $G$ by
\begin{eqnarray*}
B &=& \set{ \twomatrix{b}{0}{0}{jb^{* -1}j} \twomatrix{1_n}{aj}{0}{1_n} \in G  }
{\begin{array}{l}
 \mbox{$b$ is upper triangular of size $n$}\\
 a+a^* = 0 \end{array}}, 
\end{eqnarray*}
where $B$ consists of all the upper triangular matrices in $G$.

We introduce a spherical function $\omega(x; s)$ on $X$ by Poisson transform from relative $B$-invariants. 
For a matrix $g \in G$, denote by $d_i(g)$ the determinant of lower right $i$ by $i$ block of $g$.
Then $d_i(x), \; 1 \leq i \leq n$ are relative $B$-invariants on $X$ associated with rational characters $\psi_i$ of $B$, where
\begin{equation}
d_i(p \cdot x) = \psi_i(p)d_i(x), \quad \psi_i(p) = N_{k'/k}(d_i(p)), \quad (x \in X, \; p \in B).
\end{equation}
We set
\begin{eqnarray*}
X^{op} = \set{x \in X}{d_i(x) \ne 0, \; 1 \leq i \leq n}.
\end{eqnarray*}
%
For $x \in X$ and $s=(s_i) \in \C^n$, we consider the integral 
\begin{eqnarray} \label{def-sph}
\omega(x; s) = \int_{K}\, \abs{\bfd(k\cdot x)}^{s+\ve} dk, \quad \abs{\bfd(y)}^s = \prod_{i=1}^n\, \abs{d_i(y)}^{s_i},
\end{eqnarray}
where $dk$ is the normalized Haar measure on $K$, $k$ runs over the set $\set{k \in K}{k\cdot x \in X^{op}}$, and
\begin{equation} \label{def of ve}
\ve = \ve_0 + (\frac{\pi\sqrt{-1}}{\log q}, \ldots, \frac{\pi\sqrt{-1}}{\log q}), \quad \ve_0 = (-1,\ldots, -1, -\frac12) \in \C^n.
\end{equation}
The right hand side of (\ref{def-sph}) is absolutely convergent if $\real(s_i) \geq -\real(\ve_i) = -\ve_{0,i}, \; 1 \leq i \leq n$, and continued to a rational function of $q^{s_1}, \ldots, q^{s_n}$ (cf.~\cite[Remark 1.1]{JMSJ}), and we use the notation $\omega(x; s)$ in such sense. 
We note here that
\begin{equation*}
\abs{\psi(p)}^\ve \left( = \prod_{i=1}^n\, \abs{\psi_i(p)}^{\ve_i} \right) = \abs{\psi(p)}^{\ve_0} 
= \delta^{\frac12}(p),
\end{equation*}
where $\delta$ is the modulus character on $B$ (i.e.,  $d(pp') = \delta(p')^{-1}dp$ for the left invariant measure $dp$ on $B$). 
 
By a general theory, the function $\omega(x; s)$ becomes an $\hec$-common eigenfunction on $X$ (cf. \cite[\S 1]{JMSJ}, or \cite[\S 1]{French}), and we call it a spherical function on $X$. More precisely, the Hecke algebra $\hec$ of $G$ with respect to $K$ is the commutative $\C$-algebra consisting of compactly supported two-sided $K$-invariant functions on $G$, which acts on the space $\CKX$ of left $K$-invariant functions on $X$ by 
\begin{equation*}
(f * \Psi)(y) = \int_{G}\, f(g)\Psi(g^{-1}\cdot y) dg, \quad (f \in \hec, \; \Psi \in \CKX),
\end{equation*}
where $dg$ is the Haar measure on $G$ normalized by $\int_{K}\, dk = 1$, and we see 
\begin{eqnarray*}
(f * \omega(\; ; s))(x) = \lam_s(f) \omega(x; s), \quad (f \in \hec),
\end{eqnarray*}
where $\lam_s$ is the $\C$-algebra homomorphism defined by 
\begin{eqnarray*} \label{eigen funct}
&&
\lam_s : \hec \longrightarrow \C(q^{s_1}, \ldots, q^{s_n}), \\
&&
\quad
f \longmapsto \int_{B}\, f(p)\abs{\psi(p)}^{-s+\ve} dp.
\end{eqnarray*}

\bigskip
We introduce a new variable $z$ which is related to $s$ by
\begin{eqnarray} \label{change of var}
s_i = -z_i + z_{i+1} \quad (1 \leq i \leq n-1), \quad
s_n = -z_n 
\end{eqnarray}
and write $\omega(x;z) = \omega(x; s)$. 
Denote by $W$ the Weyl group of $G$ with respect to the maximal $k$-split torus in $B$. Then $W$ acts on rational characters of $B$ as usual (i.e., $\sigma(\psi)(b) = \psi(n_\sigma^{-1}b n_\sigma)$ by taking a representative $n_\sigma$ of $\sigma$), so $W$ acts on  $z \in \C^n$ and on $s \in \C^n$ as well.  We will determine the functional equations of $\omega(x; s)$ with respect to this Weyl group action.
The group $W$ is isomorphic to $S_n \ltimes C_2^n$, $S_n$ acts on $z$ by permutation of indices, and $W$ is generated by $S_n$ and $\tau: (z_1, \ldots, z_n) \longmapsto (z_1, \ldots, z_{n-1}, -z_n)$. Keeping the relation (\ref{change of var}), we also write $\lam_z(f) = \lam_s(f)$.
Since
$$
\abs{\psi(p)}^{-s+\ve} = \prod_{i=1}^n\, \abs{N(p_i)}^{-z_i} \times \delta^{\frac12}(p), 
$$
where $p_i$ is the $i$-th diagonal component of $p \in B$, the $\C$-algebra map $\lam_z$ is an isomorphism (the Satake isomorphism)
\begin{eqnarray} \label{satake iso}
\lam_z  &:& \hec \stackrel{\sim}{\longrightarrow} \C[q^{\pm 2z_1}, \ldots, q^{\pm 2z_n}]^W,
\end{eqnarray}
where the ring of the right hand side is the invariant subring of the Laurent polynomial ring $\C[q^{2z_1}, q^{-2z_1}, \ldots, q^{2z_n}, q^{-2z_n}]$ by $W$.

\bigskip
By using a result on spherical functions on the space of hermitian forms, we obtain the following results.

\begin{thm} \label{th: feq Sn}
The function $G_1(z) \cdot \omega(x; z)$ is invariant under the action of $S_n$ on $z$, where
\begin{eqnarray*}
G_1(z) = \prod_{1 \leq i < j \leq n}\, \frac{1 + q^{z_i-z_j}}{1 - q^{z_i-z_j-1}}.
\end{eqnarray*}
\end{thm}

\begin{proof}
By the embedding 
\begin{eqnarray*} \label{embed}
K_0 = GL_n(\calO_{k'})
\longrightarrow K, \quad h \longmapsto \wt{h} = \twomatrix{jh^{* -1}j}{0}{0}{h}, 
\end{eqnarray*}
and the normalized Haar measure $dh$ on $K_0$, we obtain, for $s \in \C^n$ satisfying $\real(s_i) \geq -\real(\ve_i), \; 1 \leq i \leq n$,
\begin{eqnarray*}
\omega(x;z) &=& \omega(x;s) = 
\dint{K_0}\,dh \dint{K}\, \abs{\bfd(k\cdot x)}^{s+\ve} dk
\nonumber \\
&=&
\dint{K_0}\,dh \dint{K}\, \abs{\bfd(\wt{h} k\cdot x)}^{s+\ve} dk
=
\dint{K}\, \dint{K_0}\, \abs{\bfd(\wt{h} k\cdot x)}^{s+\ve} dh dk\nonumber\\
&=&
\dint{K}\, {\zeta}_*^{(h)}(D(k\cdot x); z) dk. 
\label{id with herm}
\end{eqnarray*}
Here $D(k\cdot x)$ is the lower right $n$ by $n$ block of $k \cdot x$ for $\set{k \in K}{k\cdot x \in X^{op}}$, and ${\zeta}_*^{(n)}(y; z)$ is a spherical function on $\calH_n(k')$ defined by
\begin{eqnarray*}
{\zeta}_*^{(h)}(y; z) = \dint{K_0}\, \abs{\bfd(h\cdot y)}^{s+\ve} dh, \quad (h \cdot y = hyh^*),
\end{eqnarray*}
where the variable $z$ is related to $s$ by (\ref{change of var}), $\ve$ is defined in (\ref{def of ve}), and 
$h$ runs over the set $\set{h \in K_0}{d_i(h\cdot y) \ne 0, \; 1 \leq i \leq n}$. The assertion of Theorem~\ref{th: feq Sn} follows from the next proposition.
\end{proof}

%
\begin{prop} 
For any $y \in \calH_n(k')$, the function $G_1(z) \cdot {\zeta}_*^{(h)}(y; z)$ is holomorphic on $\C^n$ and invariant under the action of $S_n$.
\end{prop}

\begin{proof}
In \cite[\S 4.2]{Hamb}, we have considered the following spherical function on $\calH_n(k')$ 
\begin{eqnarray*}
\zeta^{(h)}(y; z) = \abs{\det(y)}^{\frac{n}{2}} \dint{K_0}\, \prod_{i=1}^n\, \abs{\what{d}_i(h\cdot y)}^{s_i+\ve_i} dh, 
\end{eqnarray*}
where $\what{d}_i(y)$ is the determinant of upper left $i$ by $i$ block of $y$, and the relation of $z$ and $s$ and $\ve$ are the same as before, 
and showed that the function $G_1(z) \cdot \zeta^{(h)}(y; z)$ 
is holomorphic on $\C^n$ and invariant under the action of $S_n$.
Since $d_i(y) = \det(y) \what{d}_{n-i}(y^{-1})$, we see 
\begin{eqnarray*} 
{\zeta}_*^{(h)}(y; z) &=& 
\dint{K_0}\, \abs{\det(h \cdot y)}^{\sum_{i=1}^n\,(s_i + \ve_i)} \prod_{i=1}^{n-1}\, \abs{\what{d}_{n-i}(h^{* -1}\cdot y^{-1})}^{s_i+\ve_i} dh\\
&=&
\abs{\det(y^{-1})}^{-\sum_i(s_i + \ve_i)} \cdot \dint{K_0} \prod_{i=1}^{n-1} 
   \abs{ \what{d}_i(h \cdot y^{-1})}^{s_{n-i}+\ve_i} dh\\
&=&
\zeta^{(h)}(y^{-1}; w),
\end{eqnarray*}
where $w$ is the $z$-variable corresponding to the $s$-variable $(s_{n-1}, \ldots, s_1, -(s_1+ \dots + s_n) + \frac{n}{2} + (n-1)\frac{\pi\sqrt{-1}}{\log q})$ under the relation (\ref{change of var}). Then  $w_i - w_j = z_{n-j+1} - z_{n-i+1}$ for $1 \leq i < j \leq n$, and $G_1(w) = G_1(z)$. 
Hence $G_1(z) \cdot \zeta_*^{(h)}(y; z) = G_1(w) \cdot \zeta^{(h)}(y^{-1}; w)$ is holomorphic and $S_n$-invariant.
\end{proof}

\slit
\noindent
{\bf 2.2.}
Hereafter till the end of \S \ref{sec:4}, we assume $k$ has odd residual characteristic, i.e. $q$ is odd.
In this subsection, we give the functional equation of $\omega(x;s)$ for $\tau \in W$.

\begin{thm}   \label{th: tau}
For general size $n$, the spherical function satisfies the functional equation
$$
\omega(x; z) = \omega(x; \tau(z)).
$$
\end{thm}

\bigskip
First we consider for $\omega^{(1)}(x; s)$, the case of size $n = 1$, where $z = -s$ and $\tau$ acts as $\tau(s) = -s$.
 
\begin{prop} \label{prop: size 1}
For $x_\ell = \twomatrix{\pi^{\ell}}{0}{0}{\pi^{-\ell}} \in X_1, \; \ell \geq 0$, one has
$$
\omega^{(1)}(x_\ell; s) = \frac{(-1)^\ell q^{-\frac{\ell}{2}}}{1 + q^{-1}} \times \left( 
\frac{q^{\ell s}(1 - q^{-2s-1})}{1-q^{-2s}} + \frac{q^{-\ell s}(1 - q^{2s-1})}{1-q^{2s}} \right),
$$
in particular, $\omega^{(1)}(x; s)$ is holomorphic 
on $\C$ 
and satisfies the functional equation 
$$
\omega^{(1)}(x; s) = \omega^{(1)}(x; -s).
$$
\end{prop}

\begin{proof}
It is easy to see
\begin{eqnarray*}
&&
K = K_1 = K_{1,1} \sqcup K_{1,2}, \qquad\mbox{where} \nonumber \\
&&
\hspace*{.7cm}
K_{1,1} = \set{\twomatrix{\alp}{0}{0}{\alp^{*-1}}\twomatrix{1}{v/\sqrt{\eps} } {u\sqrt{\eps}}{1+uv} }{ \alp \in \calO_{k'}^\times, \; u, v \in \calO_k}, \nonumber \\
&&
\hspace*{.7cm}
K_{1,2} = \set{ \twomatrix{\alp}{0}{0}{\alp^{*-1}}\twomatrix{\pi u \sqrt{\eps}}{1+\pi uv}{1}{v/\sqrt{\eps} } }
{ \alp \in \calO_{k'}^\times, \; u, v \in \calO_k}, \label{K1}
\end{eqnarray*}
and $vol(K_{1,1}) = \frac{1}{1+q^{-1}}$ and $vol(K_{1,2}) = \frac{q^{-1}}{1+q^{-1}}$ with respect to the measure $dh$ on $K$ normalized by $vol(K) = 1$.
Let $x_\ell$ be as above. 
For $h \in K_{1,1}$ written as above, we have
\begin{eqnarray*}
&&
d_1(h \cdot x_\ell) = N(\alp)^{-1}\pi^{-\ell}((1+uv)^2-\pi^{2\ell}u^2 \ve), \\
&&
v_\pi(d_1(h \cdot x_\ell)) = \left\{ \begin{array}{ll} 
    -\ell & \mbox{if } u \in (\pi),\\
    -\ell + 2\min\{v_\pi(1+uv), \, \ell \} & \mbox{if } u \in \calO_k^\times,
    \end{array} \right.
\end{eqnarray*}
and
\begin{eqnarray*}
\lefteqn{\int_{K_{1,1}}\, \abs{d_1(h\cdot x_\ell)}^{s-\frac12+\frac{\pi\sqrt{-1}}{\log q}} dh}\\
&=&
\frac{(-1)^\ell q^{\ell(s-\frac12)}}{1+q^{-1}} \\
&&\times
\left(  
q^{-1} + (1-q^{-1})\left(  (1-q^{-1}) + \sum_{k=1}^{\ell-1} q^{-2k(s-\frac12)}q^{-k}(1-q^{-1}) + q^{-2\ell(s-\frac12)}q^{-\ell}\right) \right)\\
&=&
\frac{(-1)^\ell q^{-\frac{\ell}{2}}}{1+q^{-1}} \left( q^{\ell s -1} + (1-q^{-1}) q^{-\ell s} + (1-q^{-1})^2 \frac{q^{\ell s} - q^{-\ell s}}{1 - q^{-2s}} \right).
\end{eqnarray*}
For $h \in K_{1,2}$ written as above, we have
\begin{eqnarray*}
&&
d_1(h \cdot x_\ell) = N(\alp)^{-1} \pi^{-\ell}(\pi^{2\ell} - v^2/\ve), \\
&&
v_\pi(d_1(h \cdot x_\ell)) = -\ell + 2\min\{v_\pi(v), \ell\},
\end{eqnarray*}
and
\begin{eqnarray*}
\lefteqn{\int_{K_{1,2}}\, \abs{d_1(h\cdot x_\ell)}^{s-\frac12+\frac{\pi\sqrt{-1}}{\log q}} dh}\\
&=&
\frac{(-1)^\ell q^{\ell(s-\frac12)}q^{-1}}{1+q^{-1}} \left(  
\sum_{k=0}^{\ell-1} q^{-2k(s-\frac12)}q^{-k}(1-q^{-1}) + q^{-2\ell(s-\frac12)}q^{-\ell} \right)\\
&=&
\frac{(-1)^\ell q^{-\frac{\ell}{2}}}{1+q^{-1}} \left( 
q^{-\ell s -1} + (q^{-1}-q^{-2}) \frac{q^{\ell s}-q^{-\ell s}}{1-q^{-2s}} \right).
\end{eqnarray*}
Hence we obtain
\begin{eqnarray*}
\omega(x_\ell; s) 
&=& 
\int_{K_1} \abs{d_1(k \cdot x_\ell)}^{s-\frac12-\frac{\pi\sqrt{-1}}{\log q}}dk\\
&=& 
\frac{1}{1+q^{-1}}\left\{   \sum_{r = 0}^\ell (-1)^\ell q^{-(2r-\ell)(s-\frac12)}q^{-r}(1-q^{-1}) 
+ (-1)^\ell q^{-\ell(s-\frac12)}q^{-(\ell+1)} \right\}\\
&& 
+ \frac{q^{-1}}{1+q^{-1}} (-1)^\ell \cdot q^{\ell(s-\frac12)}  \\
&=&
\frac{(-1)^\ell q^{-\frac{\ell}{2}}}{1 + q^{-1}} \left( 
\frac{q^{\ell s}(1 - q^{-2s-1})}{1-q^{-2s}} + \frac{q^{-\ell s}(1 - q^{2s-1})}{1-q^{2s}} \right)\\
&=&
\frac{(-1)^\ell q^{-\frac{\ell}{2}}}{1 + q^{-1}} \frac{1}{q^s-q^{-s}}\left(q^{(\ell+1)s} - q^{-(\ell+1)s} - q^{-1}(q^{(\ell-1)s}-q^{-(\ell-1)s})\right),
\end{eqnarray*}
which is holomorphic and invariant under $s \mapsto -s$. Since the set $\set{x_\ell}{\ell \geq 0}$ forms a set of complete representatives of $K_1 \backslash X_1$ (cf. Proposition~\ref{Cartan n=1}), we conclude the proof.
\end{proof}

\bigskip
We assume $n \geq 2$. 
Set 
$$
w_\tau = \begin{pmatrix}
1_{n-1} & & \\
& \twomatrixminus{0}{1}{1}{0} & \\
& & 1_{n-1}
\end{pmatrix},
$$
and take the standard parabolic subgroup $P$ attached to $\tau$
\begin{align} 
\label{P-tau}
\lefteqn{P =  B \cup B w_\tau B }  \\
&= \notag
\left\{ \left. 
\left( \begin{array}{cc|cc} 
      {q'} & & & \\
      & a & b & \\
      \hline
      & c & d & \\
      & & & q
\end{array} \right)
\left( \begin{array}{cc|cc} 
       1_{n-1} & \alp & & \\
               & 1 & 0 &  \\
       \hline
       &  0 & 1 & -\alp^*j\\
       & & & 1_{n-1}
\end{array} \right) 
\left( \begin{array}{cc|cc} 
       1_{n-1} & & \beta & \gamma j\\
               & 1 & 0 & -\beta^* j \\
       \hline
       &  0 & 1 & \\
       & & & 1_{n-1}
\end{array} \right) \in G
\right\vert \right. \\
& \hspace*{1cm} \notag
\left. \vphantom{
 \twomatrixplus
 {\twomatrixminus{1_{n-1}}{\alp}{}{1}} {} {}
 {\twomatrixminus{1_{n-1}}{}{-\alp^*}{1}}
} 
\begin{array}{l} 
q \mbox{ is upper triangular in }GL_{n-1}(k'), \; q' = jq^{* -1}j\\
\twomatrix{a}{b}{c}{d} \in U(j_2), \;
 \alpha, \beta \in M_{n-1,1}(k'),\\
 \gamma \in M_{n-1}(k'), \; \gamma + \gamma^* = 0
  \end{array}
 \right\},  
\end{align}
where $j = j_{n-1}$ and each empty place in the above expression means zero-entry.

We consider the following action of $P' = P \times GL_1(k')$ on $X' = X \times V$ with $V = M_{21}(k')$:
$$
(p,r) \star (x,v) = (p\cdot x, \rho(p)vr^{-1}), \quad (p,r) \in P', \; (x,v) \in X',
$$
where $\rho(p) = \twomatrix{a}{b}{c}{d}$ for the decomposition of $p \in P$ as in (\ref{P-tau}).
We define
\begin{equation} \label{g(x,v)}
g(x,v) = \det\left[ \twomatrixplus{-v_2\; v_1}{0}{0}{1_{n-1}}\cdot x_{(n+1)} \right],\qquad (x, v) \in X', \; v = \twovector{v_1}{v_2}, 
\end{equation}
where $x_{(n+1)}$ is the lower right $(n+1)$ by $(n+1)$ block of $x$.
Then we have the following.

\begin{lem}
Let $g(x,v)$ be the function on $X'$ defined by (\ref{g(x,v)}).

{\rm (i)} $g(x,v)$ is a $P'$-relative invariant on $X'$ associated with the $P'$-rational character $\wt{\psi}(p,r) = \psi_{n-1}(p)N(r)^{-1}$, and $g(x,v_0) = d_n(x)$ with $v_0 = {}^t(1 \, 0)$.

{\rm (ii)} $g(x, v)$ is expressed as 
$g(x,v) = D(x)[v]$
by some hermitian matrix $D(x)$ of size $2$. For $x \in X^{op}$, $D_1(x) = d_{n-1}(x)^{-1} D(x) $ belongs to $X_1$.
\end{lem}

\begin{proof}
(i) It is easy to see that $g(x, v_0) = d_n(x)$ and $g((1,r)\star(x,v)) = N(r)^{-1}g(x,v)$. 
Take an element $p$ in $P$ and write as  
$$
p = \left( \begin{array}{c|c|c} 
q' & \alp' & \gamma \\
\hline
0 & \rho(p) & \alp\\
\hline
0 & 0 & q
\end{array} \right), \quad (q, q', \gamma \in M_{n-1}, \; \alp,  {}^t\alp' \in M_{2,n}, \; \rho(p) = \twomatrix{a}{b}{c}{d} \in U(j_2)).
$$
Then
\begin{eqnarray*}
g((p,1)\star(x,v)) &=& \det\left[ 
\twomatrixplus{(-v_2\, v_1)\twomatrix{d}{-b}{-c}{a}} {0}{0}{1_{n-1}}\cdot
\twomatrixplus{\rho(p)}{\alp}{0}{q} \cdot x_{(n+1)}
 \right]\\
 &=&
\det\left[ 
\twomatrixplus{u(-v_2\, v_1)}{\beta}{0}{q} \cdot x_{(n+1)} \right]\\
&=&
\det\left[ \twomatrixplus{u}{\beta}{0}{q} \cdot \twomatrixplus{-v_2\; v_1}{0}{0}{1_{n-1}}\cdot x_{(n+1)} \right]\\
&=&
N(\det(q))g(x,v) \\
&=& 
\psi_{n-1}(p)g(x,v),
\end{eqnarray*}
where $u = \det(\rho(p)) \in \calO_{k'}^1 \left( = \set{u \in \calO_{k'}^\times}{N(u) = 1} \right)$ and $\beta = (-v_2\, v_1)\twomatrix{d}{-b}{-c}{a} \alp \in M_{1n}(k')$.
Hence we see that
\begin{eqnarray} \label{P-action}
g((p, r) \star (x, v)) = \psi_{n-1}(p)N(r)^{-1} g(x,v), \quad (p,r) \in P'.
\end{eqnarray}

\medskip
(ii) Since $g(x,v)$ is a linear form with respect to $v_1, v_2$ and $v_1^*, v_2^*$ and $g(x,v)^* = g(x,v)$, it is written as $D(x)[v]$ for some hermitian matrix $D(x)$ of size $2$. 
For $x = Diag(a_1^{-1}, \ldots, a_n^{-1}, a_n, \ldots, a_1) \in X^{op}$, we have
\begin{eqnarray*}
g(x, v) &=& 
\det\left[ 
\twomatrixplus{-v_2\; v_1}{0}{0}{1_{n-1}}\cdot Diag(a_n^{-1}, a_n, \ldots, a_1) \right]\\
&=&
(a_1 \cdots a_n) v_1v_1^* + (a_1\cdots a_{n-1}a_n^{-1}) v_2v_2^*.
\end{eqnarray*}
Hence, for any diagonal $x \in X^{op}$, we have 
\begin{eqnarray}
D(x) &=& \twomatrix{d_n(x)}{0}{0}{d_n(x)^{-1}d_{n-1}(x)^2} \nonumber \\
D_1(x) &=& d_{n-1}(x)^{-1} D(x) \in X_1. \label{diag x}
\end{eqnarray}
For any $x \in X^{op}$, we have $x = b\cdot y$ for some diagonal $y \in X^{op}$ and $b \in B$.
Since
\begin{eqnarray*}
(x, v) = (b, 1)\star (y, \rho(b)^{-1}v),
\end{eqnarray*}
we have by (\ref{P-action}) and (\ref{diag x}),
\begin{eqnarray*}
D(x) &=& \psi_{n-1}(b) \left( \rho(b)^{* -1} \cdot D(y)  \right) = \psi_{n-1}(b)d_{n-1}(y) \left( \rho(b)^{* -1} \cdot D_1(y)  \right)\\
& = &
d_{n-1}(x) \left( \rho(b)^{* -1} \cdot D_1(y) \right),
\end{eqnarray*}
and $D_1(x) = \rho(b)^{* -1} \cdot D_1(y) \in X_1$, since $\rho(b)^{* -1} \in G_1 = U(j_2)$.
\end{proof}

\bigskip
\begin{proof}[Proof of Theorem~\ref{th: tau}]
By the embedding
\begin{eqnarray}
K_1 = U(j_2) \hookrightarrow K = K_n, \; h \longmapsto \wt{h} = 
\begin{pmatrix}
1_{n-1} & & \\
& h & \\
& & 1_{n-1}
\end{pmatrix},
\end{eqnarray}
we see
\begin{eqnarray}
\omega(x;s) 
&=& 
\int_{K_1}dh \int_{K} \abs{\bfd(k\cdot x)}^{s+\ve}dk \nonumber \\
&=& 
\int_{K_1}dh \int_{K} \abs{\bfd(\wt{h}k\cdot x)}^{s+\ve}dk \nonumber \\
&=&
\int_{K}\, \prod_{i<n}\, \abs{d_i(k\cdot x)}^{s_i + \ve_i} \int_{K_1}\,\abs{d_n(\wt{h}k \cdot x)}^{s_n+\ve_n} dh dk. \label{first stage}
\end{eqnarray}
For $y \in X^{op}$, we have
\begin{eqnarray*}
d_n(\wt{h}\cdot y) 
&=& g(\wt{h}\cdot y, v_0) = g((\wt{h},1) \star (y, h^{-1}v_0)) \\
&=&
g(y, h^{-1}v_0) \qquad (\mbox{since }\; \psi_{n-1}(\wt{h}) = 1)\\
&=&
D(y)[h^{-1}v_0] = \what{d_1}(h^{* -1}\cdot D(y)) = d_{n-1}(y) \what{d_1}(h^{* -1} \cdot D_1(y)),
\end{eqnarray*}
where $\what{d_1}(\cdot)$ is the $(1,1)$-component of $\cdot$\,. Since $\what{d_1}(x_1) = d_1(x_1^{-1})$ for $x_1 \in X_1$, we have
\begin{eqnarray*}
d_n(\wt{h}\cdot y) &=&
d_{n-1}(y) d_1( (h^{* -1} \cdot D_1(y) )^{-1} ) = d_{n-1}(y)d_1(h \cdot D_1(y)^{-1}).
\end{eqnarray*}
Returning to (\ref{first stage}), we have
\begin{eqnarray*}
\omega(x;s) &=& 
\int_{K}\, \prod_{i<n}\, \abs{d_i(k\cdot x)}^{s_i + \ve_i} \int_{K_1}\,\abs{d_{n-1}(k\cdot x) d_1(h \cdot D_1(k\cdot x)^{-1})}^{s_n+\ve_n} dh dk\\
&=&
\int_{K}\, \prod_{i<n}\, \abs{d_i(k\cdot x)}^{s_i + \ve_i} \abs{d_{n-1}(k\cdot x)}^{s_n+\ve_n} \omega^{(1)}(D_1(k\cdot x)^{-1}; s_n) dk, 
\end{eqnarray*}
where $\omega^{(1)}(y; s)$ is the spherical function of size $n = 1$. Then, by Proposition~\ref{prop: size 1}, we obtain
\begin{eqnarray*}
\omega(x;s)  
&=&
\omega(x; s_1, \ldots, s_{n-2}, s_{n-1}+2s_n, -s_n),
\end{eqnarray*}
which shows in $z$-variable
\begin{eqnarray*}
\omega(x;z) = \omega(x; \tau(z)), \quad \tau(z) = (z_1, \ldots, z_{n-1}, -z_n),
\end{eqnarray*}
and we conclude the proof.  
\end{proof}

\slit
\noindent
{\bf 2.3.}
Since our spherical function $\omega(x;s)$ satisfies the same functional equations with respect to $S_n$ and $\tau$ (Theorem~\ref{th: feq Sn} and Theorem~\ref{th: tau}) as $\omega_T(x;s)$ in \cite{Oda} does, we have the same functional equations with respect to $W$ also. For the proof, we may follow a similar line as in \cite[\S 2.3]{Oda}, so we omit the details. 

We denote by $\Sigma$ the set of roots of $G$ with respect to the maximal $k$-split torus of $G$ contained in $B$ and by $\Sigma^+$ the set of positive roots with respect to $B$. We may understand $\Sigma$ as a subset in $\Z^n$, and set 
\begin{eqnarray}
&&
\Sigma^+ = \Sigma^+_s \cup \Sigma^+_\ell, \label{short and long}\\
&&
\Sigma^+_s = \set{e_i - e_j, \; e_i + e_j}{1 \leq i < j \leq n}, \quad \Sigma^+_\ell = \set{2e_i}{1 \leq i \leq n},\nonumber
\end{eqnarray}
where $e_i$ is the $i$-th unit vector in $\Z^n, \; 1 \leq i \leq n$.
%
We define a pairing on $\Z^n \times \C^n$ by
$$
\pair{t}{z} = \sum_{i=1}^n t_iz_i, \qquad (t \in \Z^n, \; z \in \C^n),
$$
which satisfies 
\begin{eqnarray*} \label{W-invariance}
\pair{\alp}{z} = \pair{\sigma(\alp)}{\sigma(z)}, \qquad (\alp \in \Sigma, \; z \in \C^n, \; \sigma \in W).
\end{eqnarray*}

\bigskip
\begin{thm} \label{th: feq}
The spherical function $\omega(x;z)$ satisfies the following functional equation
\begin{eqnarray} \label{Gamma-sigma}
\omega(x; z) = \Gamma_\sigma(z) \cdot \omega(x; \sigma(z)),
\end{eqnarray}
where 
\begin{eqnarray*} \label{sigma-factors}
\Gamma_\sigma(z) = \dprod{\alp \in \Sigma^+_s(\sigma)}\,
\frac{1 - q^{\pair{\alp}{z}-1}}{q^{\pair{\alp}{z}} - q^{-1}}, \quad
\Sigma^+_s(\sigma) = \set{\alp \in \Sigma^+_s}{-\sigma(\alp) \in \Sigma^+}. 
\end{eqnarray*}
\end{thm}

\bigskip
The next theorem can be proved in a similar line to the proof of \cite[Theorem~2.9]{Oda}.

\bigskip
\begin{thm} \label{th: W-inv} 
The function $G(z) \cdot \omega(x; z)$ is holomorphic
 on $\C^n$ 
and $W$-invariant, in particular it is an element in $\C[q^{\pm z_1}, \ldots, q^{\pm z_n}]^{W}$, where
\begin{eqnarray*}
G(z) = \prod_{\alp \in \Sigma^+_s}\, \frac{1 + q^{\pair{\alp}{z}}}{1 - q^{\pair{\alp}{z}-1}}.
\end{eqnarray*}
\end{thm}

\bigskip
\noindent
\begin{proof}[An outline of a proof] \quad
It is clear that $G(z) \cdot \omega(x; z)$ is invariant under the action of $\sigma \in \set{(i\; i+1) \in S_n}{1 \leq i \leq n-1} \cup \{ \tau \}$ by Theorem~\ref{th: feq}, hence it is invariant for every $\sigma \in W$ by cocycle relations of Gamma factors. 

In order to prove the holomorphy, we consider the following integral for any compactly supported function $\vphi$ in $\CKX$:
\begin{eqnarray*}
\Phi(z; \vphi) =  \int_{X^{op}}\, \vphi(x) \abs{\bfd(x)}^{s+\ve} dx, 
\end{eqnarray*}
where $dx$ is a $G$-invariant measure on $X$ and the relation of $z$ and $s$ and $\ve$ are the same as before. Then, the right hand side is absolutely convergent if $\real(s_i) \geq 1, (1 \leq i \leq n-1)$ and $\real(s_n) \geq \frac12$, and continued to a rational function in $q^{s_1}, \ldots, q^{s_n}$. 
Taking $\vphi$ to be the characteristic function of $K\cdot x$, we see 
$\Phi(z; \vphi) = vol(K\cdot x) \cdot \omega(x; z)$, hence it is enough to show the holomorphy of $G(z) \cdot \Phi(z; \vphi)$. 
To begin with, $\Phi(z; \vphi)$ is holomorphic on    
\begin{eqnarray*}
\calD_0 &=& 
\set{z\in \C^n}{-\frac12 \geq \real(z_n), \; \real(z_{i+1}) \geq \real(z_i)+1, \; (1 \leq i \leq n-1)},
\end{eqnarray*}
and satisfies the same functional equations
\begin{eqnarray*}
\Phi(z; \vphi) = \Gamma_\sigma(z) \Phi(\sigma(z); \vphi), \quad \sigma \in W.
\end{eqnarray*}
Since $G(z)$ is holomorphic on $\calD_0$, $G(z) \cdot \Phi(z; \vphi)$ is holomorphic on  
\begin{eqnarray*}
\bigcup_{\sigma \in W}\, \sigma(\calD_0).
\end{eqnarray*}
We recall $G_1(z)$ in Theorem~\ref{th: feq Sn} and write  
$$
G(z) = G_1(z) \times G_2(z),\quad G_2(z) = \prod_{1 \leq i < j \leq n}\, \frac{1+q^{z_i+z_j}}{1 - q^{z_i+z_j-1}}.
$$
We obtain, in a similar way to the proof of Theorem~\ref{th: feq Sn},
$$
\Phi(z; \vphi) = \int_{X^{op}}\, \vphi(x) \zeta^{(h)}_*(D(x); z) dx,
$$
then we see 
$G(z) \cdot \Phi(z; \vphi)$ is holomorphic on
$$
\calD_1 = \set{z \in \C^n}{\real(z_i+z_j) \ne 1, \; (1 \leq i < j \leq n)},
$$
since $G_2(z)$ is holomorphic on $\calD_1$ and $\vphi$ is compactly supported.
We see $G(z) \cdot \Phi(z; \vphi)$ is holomorphic on $\C^n$, since it is holomorphic on  
$$
\wt{\calD} = \bigcup_{\sigma \in W} \sigma(\calD_0 \cup \calD_1),
$$
and the convex hull of the connected set $\wt{\calD}$ is $\C^n$. 
\end{proof}

\vspace{2cm}

\section{The explicit formula for $\omega(x;z)$}
\label{sec:3}

{\bf 3.1.} 
We give the explicit formula of $\omega(x; z)$. Since $\omega(x; z)$ is stable on each $K$-orbit, it is enough to show the explicit formula for each $x_\lam, \; \lam \in \Lam_n^+$ by Theorem~\ref{thm: Cartan}.

\begin{thm} \label{th: explicit}
For $\lam \in \Lam_n^+$, one has the explicit formula:
\begin{eqnarray*} 
\omega(x_\lam; z)
 &=&
\frac{(1-q^{-2})^n}{w_{2n}(-q^{-1})} \cdot \frac{1}{G(z)} \cdot c_\lam \cdot Q_\lam(z),
\end{eqnarray*}
where $G(z)$ is given in Theorem~\ref{th: W-inv}, and 
\begin{eqnarray}
&&
w_m(t) = \prod_{i=1}^m (1 - t^i),\qquad
c_\lam = (-1)^{\sum_i \lam_i (n-i+1)} q^{-\sum_i \lam_i (n-i + \frac12)},\nonumber \\
&&
Q_\lam(z) = 
\sum_{\sigma \in W}\, \sigma\left(  q^{-\pair{\lam}{z}} c(z) \right), \nonumber\\
&& \label{c(z)}
c(z) = \prod_{\alp \in \Sigma_s^+}\, \frac{1 + q^{\pair{\alp}{z}-1}}{1 - q^{\pair{\alp}{z}}}
\prod_{\alp \in \Sigma_\ell^+}\, \frac{1 - q^{\pair{\alp}{z}-1}}{1 - q^{\pair{\alp}{z}}}. 
\end{eqnarray}
\end{thm}

\begin{rem} \label{rem Mac}
The above formula is the same as the explicit formula of $\omega_T(y_\lam; z)$ on $X_T$ at $y_\lam \in X_T$ parametrized by $\lam \in \Lam_n^+$ in \cite[Theorem~3.3]{Oda}. We explain the relation between the spaces $X$ and $X_T$'s in Appendix \ref{sec:appC}. 

We see that the main part $Q_\lam(z)$ of $\omega(x_\lam; z)$ belongs to $\calR = \C[q^{\pm z_1}, \ldots, q^{\pm z_n}]^W$ by Theorem~\ref{th: W-inv}. 
On the other hand 
$Q_\lam(z)$ is a Hall-Littlewood polynomial $P_\lam$ of type $C_n$ up to constant multiple, which is introduced in a general context of orthogonal polynomials associated with root systems (\cite[\S 10]{Mac}), and $Q_{\bf0}(z)$ is a specialization of Poincar{\'e} polynomial (\cite[Th.2.8]{Mac2}). More precisely, 
\begin{eqnarray} \label{HL-Q}
&&
Q_\lam(z) = \frac{\wt{w_\lam}(-q^{-1})}{(1+q^{-1})^n} \cdot P_\lam(z), \\
&& 
\wt{w_{\lam}}(t) = w_{m_0(\lam)}(t)^2 \cdot  \prod_{\ell \geq 1}\, w_{m_\ell(\lam)}(t), \quad 
m_\ell(\lam) = \sharp\set{i}{\lam_i = \ell}, \nonumber
\end{eqnarray}
and it is known that the set $\set{Q_\lam(z)}{\lam \in \Lam_n^+}$
forms a $\C$-basis for $\calR$, and in particular, $Q_{\bf0}(z)$ is a
constant independent of $z$.  For details see Appendix
\ref{sec:appB}.
\end{rem}

\bigskip
By Theorem~\ref{th: explicit} and Remark~\ref{rem Mac}, we have the following corollary. 

\begin{cor} \label{cor}
For $x_{\bf0} = 1_{2n}$, one has
$$
\omega(1_{2n}; z) = 
\frac{(1-q^{-1})^n w_n(-q^{-1})^2}{w_{2n}(-q^{-1})} \times
\frac{1}{G(z)}.
$$
\end{cor}

\begin{proof} 
By definition, we see that $\omega(x; -\ve) = 1$ with $s$-variable $-\ve$. We denote by $z^*$ the value in $z$-variable corresponding to $-\ve$. Since $Q_{\bf0}(z)$ is independent of $z$, we have
\begin{eqnarray*}
Q_{\bf0}(z) =  Q_{\bf0}(z^*) = \left\{ \frac{(1-q^{-2})^n}{w_{2n}(-q^{-1})} \cdot \frac{1}{G(z^*)}  \right\}^{-1} =
\frac{w_n(-q^{-1})^2}{(1+q^{-1})^n},
\end{eqnarray*}
and the result follows from this.
\end{proof}

\bigskip

We will prove Theorem \ref{th: explicit}
 by using a general expression formula given in \cite{French} (or in \cite{JMSJ}) of spherical functions on homogeneous spaces, which is based on functional equations of finer spherical functions and some data depending only on the group $G$. 
We need to check the assumptions there. Let $\G$ be a connected reductive linear algebraic group and $\X$ be a $\G$-homogeneous affine algebraic variety, where everything is assumed to be defined over a $p$-adic field $k$. For an algebraic set, we use the same ordinary letter to indicate the set of $k$-rational points. Let $K$ be a special good maximal compact open subgroup of $G$, and $\B$ a minimal parabolic subgroup of $\G$ defined over $k$ satisfying $G = KB = BK$. 
We denote by $\frX(\B)$ the group of rational characters of $\B$ defined over $k$ and by $\frX_0(\B)$ the subgroup consisting of those characters associated with some relative $\B$-invariant on $\X$ defined over $k$. In this situation, the assumptions are the following:

\mslit
$(A1)$ $\X$ has only a finite number of $\B$-orbits (, hence there is only one open orbit $\X^{op}$).

\mslit
$(A2)$ A basic set of relative $\B$-invariants on $\X$ defined over $k$ can be taken by regular functions on $\X$.

\mslit
$(A3)$ For $y \in \X \backslash \X^{op}$, there exists some $\psi$ in  $\frX_0(\B)$ whose restriction to the identity component of the stabilizer $\B_y$ of $\B$ at $y$ is not trivial.

\mslit
$(A4)$ The rank of $\frX_0(\B)$ coincides with that of $\frX(\B)$.

\slit
In the present situation, our space $X$ is isomorphic to $U(j_{2n})/U(j_{2n})\cap U(1_{2n})$ over $\ol{k}$ 
(cf.~Appendix~\ref{sec:appA} and (\ref{space X})), 
which is a symmetric space and $(A1)$ is satisfied. $(A2)$ and $(A4)$ are satisfied by our relative $B$-invariants $\set{d_i(x)}{1 \leq i \leq n}$, where $n$ is the rank of $\frX_0(\B) = \frX_0(\B)$ and 
$\X^{op} = \set{x \in \X}{d_i(x) \ne 0, \; 1 \leq i \leq n}$.    To check $(A3)$ is crucial and rather complicated. 

\mslit
  We admit the condition $(A3)$ for a while, which is proved in \S 3.2, and prove Theorem~\ref{th: explicit}. 
The set $X^{op} = \set{x \in X}{d_i(x) \ne 0, \; 1 \leq i \leq n}$ is decomposed into the disjoint union of $B$-orbits as follows:
\begin{eqnarray*} \label{factor in frX-op}
&&
X^{op} = \dsqcup{u \in \calU}\, X_u, \qquad \calU = \left( \Z/2\Z \right)^n, \\
&&
X_u = \set{x \in X^{op}}{v_\pi(d_i(x)) \equiv u_1 + \cdots + u_i\pmod{2}, \; 1 \leq i \leq n}.\nonumber
\end{eqnarray*}
According to the decomposition of $X^{op}$, we consider finer spherical functions 
$$
\omega_u({x}; s) = \dint{K}\, \abs{\bfd(k\cdot x)}_u^{s+{\ve}}dk, \quad
\abs{\bfd(y)}_u^s = \left\{ \begin{array}{ll} 
\prod_{i=1}^n \abs{d_i(y)}^{s_i} & \mbox{if } {y} \in X_u,\\
{} & {}\\
0 & \mbox{otherwise .}
\end{array}
\right.
$$
Then, for ``generic $z$'', the set $\set{\omega_u(x, z)}{u \in \calU}$ becomes a basis for the space of spherical functions on $X$ associated with the same $\lam_z$, where we keep the relation (\ref{change of var}) between $s$ and $z$. 
Here ``generic $z$'' means that $f(q^{z_1}, \ldots, q^{z_n}) \ne 0$ for a polynomial $f(x_1, \dots, x_n)$, which comes from the bijectivity of Poisson integral (cf.~\cite[Theorem~3.2]{Kato}) and the condition $(A3)$ (cf.~\cite[Lemma~2.4]{French}).
For each character $\chi$ of $\calU$, we may represent as follows
\begin{eqnarray} \label{omega(z_chi)}
\sum_{u \in \calU}\, \chi(u) \omega_u(x;z) = \omega(x; z_\chi),
\end{eqnarray}
where $z_\chi$ is obtained by adding $\dfrac{\pi\sqrt{-1}}{\log q}$ to $z_i$ for suitable $i$ according to $\chi$, and they are linearly independent (for generic $z$) as varying characters $\chi$. 
By the functional equation of $\omega(x;z)$ (Theorem~\ref{th: feq}), 
we have for each $\sigma \in W$
\begin{eqnarray} 
\omega(x; z_\chi) &=& 
\Gamma_\sigma(z_\chi) \omega(x; \sigma(z_\chi))\nonumber \\
&=&
\Gamma_\sigma(z_\chi) \omega(x; \sigma(z)_{\sigma(\chi)}), \label{feq for omega-chi}
\end{eqnarray}
by taking a suitable character $\sigma(\chi)$ of $\calU$. 
When $\chi$ is the trivial character ${\bf 1}$, the equation (\ref{feq for omega-chi}) coincides with the original functional equation of $\omega(x; z)$ and $\Gamma_\sigma(z_{\bf 1}) = \Gamma_\sigma(z)$.
By (\ref{omega(z_chi)}) and (\ref{feq for omega-chi}), we obtain vector-wise functional equations for finer spherical functions $\omega_u(x;z)$ 
\begin{equation} \label{matrix-feq}
\left( \omega_u(x; z) \right)_{u \in \calU} = A^{-1} \cdot G(\sigma,z)\cdot \sigma A \cdot \left( \omega_u(x; \sigma(z)) \right)_{u \in \calU}, \qquad \sigma \in W,
\end{equation}
where
\begin{eqnarray*}
A = (\chi(u))_{\chi, u}, \quad
\sigma A = (\sigma(\chi)(u))_{\chi, u} \in GL_{2^n}(\Z),
\end{eqnarray*}
$\chi$ runs over characters of $\calU$, $u \in \calU$, and $G(\sigma, z)$ is the diagonal matrix of size $2^n$ whose $(\chi, \chi)$-component is $\Gamma_\sigma(z_\chi)$. 
We denote by $U$ the Iwahori subgroup of $K$ compatible with $B$ and take the normalized Haar measure $du$ on $U$;
$$
U = \set{\nu=(u_{ij}) \in K}{\begin{array}{l}
 u_{ii} \in \calO_{k'}^\times \; \mbox{for } 1 \leq  i \leq 2n, \\
 u_{ij} \in \pi\calO_{k'} \; \mbox{if } i > j
 \end{array} }.  
$$ 
Then it is easy to see,  for any $\nu \in U$ and $x_\lam$ with $\lam \in \Lam_n^+$
\begin{eqnarray*}
\abs{d_i(\nu\cdot x_\lam)} = \abs{d_i(Diag(\pi^{-\lam_n}, \ldots, \pi^{-\lam_1}))} = q^{(\lam_1 + \cdots + \lam_i)},
\end{eqnarray*}
which means $x_\lam \in \calR^+$ in the sense of \cite[(2.8)]{French}.
We set
\begin{eqnarray*} 
\delta_{u}(x_\lam, z) &=& \dint{U}\, \abs{\bfd(\nu \cdot x_\lam)}_u^{s+{\ve}} d\nu.
\end{eqnarray*}
Then we have 
\begin{equation*}
\delta_u(x_\lam,z) =  \left\{ \begin{array}{ll} 
\abs{\bfd(x_\lam)}^{s+{\ve}} & \mbox{if } x_\lam  \in X_u\\
{} & {}\\
0 & \mbox{otherwise}.
\end{array}
\right\} 
=
\label{delta-u}
 \left\{ \begin{array}{ll} 
c_\lam q^{-\pair{\lam}{z}} & 
\mbox{if } x_\lam  \in X_u\\
{} & {}\\
0 & \mbox{otherwise}.
\end{array}
\right.
\end{equation*}
Applying \cite[Theorem~2.6]{French} to our present case, we obtain for generic $z$, by virtue of (\ref{matrix-feq}),
\begin{equation} \label{gen-formula}
\left( \omega_u(x_\lam; z) \right)_{u \in \calU} = 
\frac{1}{Q} \sum_{\sigma \in W}\, \gamma(\sigma(z)) \left( A^{-1} \cdot G(\sigma, z) \cdot \sigma A \right) \left( \delta_u(x_\lam, \sigma(z)) \right)_{u \in \calU},
\end{equation}
where
\begin{eqnarray}
&&
Q = \sum_{\sigma \in W}\, [U\sigma U : U]^{-1} = \frac{w_{2n}(-q^{-1})}{(1-q^{-2})^n}, \nonumber \\
&&  \label{gamma(z)}
\gamma(z) =  \prod_{\alp \in \Sigma_s^+}\, \frac{1 - q^{2\pair{\alp}{z}-2}}{1 - q^{2\pair{\alp}{z}}}
\cdot
\prod_{\alp \in \Sigma^+_\ell}\, \frac{ 1 - q^{\pair{\alp}{z}-1} }{ 1 - q^{\pair{\alp}{z}}}. 
\end{eqnarray}
Then 
we obtain
\begin{eqnarray*}
\omega({x_\lam}; z) &=& \sum_{u \in \calU}\, {\bf 1}(u) \omega_u(x_\lam; z)\\
&=&
\mbox{{\rm the $(\chi={\bf 1})$-entry of }} A \left( \omega_u(x_\lam; z) \right)_{u \in \calU}\\
&=&
\mbox{{\rm the $(\chi={\bf 1})$-entry of }} \frac{1}{Q} \sum_{\sigma \in W}\, \gamma(\sigma(z)) \left( G(\sigma, z) \cdot \sigma A \right) \left( \delta_u(x_\lam, \sigma(z)) \right)_{u \in \calU}\\
&=&
\frac{(1-q^{-2})^n }{w_{2n}(-q^{-1})} \times \sum_{\sigma \in W}\, \gamma(\sigma(z)) \Gamma_\sigma(z) \sum_{u}\delta_u(x_\lam, \sigma(z))\\
&=&
\frac{c_\lam (1-q^{-2})^n }{w_{2n}(-q^{-1})} \times \sum_{\sigma \in W}\, {\gamma(\sigma(z))} \Gamma_\sigma(z)  q^{-\pair{\lam}{\sigma(z)}}.
\end{eqnarray*}
By Theorem~\ref{th: feq}, Theorem~\ref{th: W-inv}, (\ref{c(z)}) and (\ref{gamma(z)}),  we have
\begin{eqnarray*}
\Gamma_\sigma(z) = \frac{G(\sigma(z))}{G(z)}, \qquad \gamma(z) \cdot G(z) = c(z).
\end{eqnarray*}
Thus we obtain the required explicit formula of $\omega(x_\lam; z)$ for generic $z$, and it is also valid for any $z\in \C^n$, since $G(z)\cdot \omega(x_\lam;z)$ is a polynomial in $q^{\pm z_1}, \ldots, q^{\pm z_n}$. 
\hfill\qedsymbol

\bigskip
\noindent
{\bf 3.2.} 
In this subsection, we prove the space $X_n$ satisfies the condition $(A3)$ by induction on $n$. 
For $n = 1$, the condition $(A3)$ is obvious, since $X_1 = X_1^{op}$ (cf.~Proposition~\ref{Cartan n=1}).
Hereafter we assume $n \geq 2$. 
We set
\newcommand{\bfb}{{\bf b}}
\begin{eqnarray*}
&&
t(\bfb) = Diag(b_1, \ldots, b_n, b_n^{-1}, \ldots, b_1^{-1}) \in B (= B_n), \qquad \bfb = (b_1, \ldots, b_n) \in (k^{\times})^n,\\
&&
B_0 = \set{\twomatrix{j_nb^{*-1}j_n}{0}{0}{b} \in B}{b \in GL_n(k'), \; \mbox{upper triangular}}, \\
&&
N_B = \set{\twomatrix{1_n}{A}{0}{1_n} \in B}{jA + A^*j = 0}.
\end{eqnarray*}


\medskip
\begin{lem} \label{lem: d1 ne 0}
Assume $x \in X_n$ and $d_1(x) \ne 0$.
Then, the orbit $B \cdot x$ contains an element of type 
$$
\left( \begin{array}{c|c|c}
a^{-1} & 0 & 0\\
\hline
0 & y & 0\\
\hline
0 & 0 & a
\end{array} \right), \qquad  a = 1, \pi, \quad y \in X_{n-1}.
$$
Here, if $x \notin X_n^{op}$, then $y \notin X_{n-1}^{op}$. 
\end{lem}

\begin{proof}
By the action of $B_0$, we may assume $x = (x_{i,j})$ satisfies
$$
x_{2n,2n} = 1, \pi, \quad x_{2n,j} = x_{j,2n} = 0, \; (n+1 \leq j \leq 2n-1).
$$
Then, by the action of $N_B$, we may change 
$$
x_{2n,j} = x_{j,2n} = 0, \; (2 \leq j \leq n), \quad x_{2n,1} = x_{1,2n} \in k.
$$
Then, by the property $x^* = x$ and $j_{2n}[x] = j_{2n}$, we see 
$$
x_{1,j} = x_{j,1} = 0, \; (j \geq 2), \quad x_{1, 1} = {x_{2n,2n}}^{-1},
$$
hence may assume $x$ has the shape as in the statement, and $\Phi_{xj_{2n}}(t) = (t^2-1) \Phi_{yj_{2n-2}}(t)$, hence $y \in X_{n-1}$. The second assertion is clear. 
\end{proof}

\medskip
\begin{lem} \label{lem: lower}
Assume $x \in X_n$, $d_1(x) = 0$ and the first non-zero entry in the $2n$-th column from the bottom stands at $(2n-\ell+1, 2n)$ with $1 < \ell \leq  n$.
Then there is some $y = (y_{i,j}) \in B\cdot x$ which satisfies
\begin{eqnarray*}
&&
y_{1,j} = y_{j,1} = \delta_{j, \ell}, \quad y_{2n, j} = y_{j, 2n} = \delta_{2n-\ell+1, j}, \\
&&
y_{\ell, j} = y_{j, \ell} = \delta_{1, j}, \quad 
y_{2n-\ell+1, j} = y_{j, 2n-\ell+1} = \delta_{2n, j},
\end{eqnarray*}
where $\delta_{i,j}$ is the Kronecker delta.

\noindent
The stabilizer $B_y$ contains $t(\bfb)$ with $b_1 = b_\ell^{-1} = b \in k^\times$ and the remaining $b_i$ being $1$, and the character $\psi_1$ is not trivial on $B_y$.
\end{lem}

\begin{proof}
By the action of $B_0$, we may assume $x = (x_{i,j})$ satisfies
\begin{eqnarray*}
&&
x_{2n-\ell+1, j} = x_{j, 2n-\ell+1} = \left\{ \begin{array}{ll} 
        1 & \mbox{if } j = 2n \\ 0 & \mbox{if } n+1 \leq j \leq 2n-1,
\end{array} \right.\\
&&
x_{2n, j} = x_{j, 2n} = \left\{ \begin{array}{ll} 
        1 & \mbox{if } j = 2n-\ell+1 \\  0 & \mbox{if } n+1 \leq j \leq 2n, \, j \ne 2n-\ell+1.
\end{array} \right.
\end{eqnarray*}
Then by the action of $N_B$, we may take $y \in B\cdot x$ such that
\begin{eqnarray*}
&&
y_{2n, j} = y_{j, 2n} = \left\{ \begin{array}{ll} 
        1 & \mbox{if } j = 2n-\ell+1 \\ 
        a & \mbox{if } j = \ell\\ 
        0 & \mbox{if } j \ne \ell, 2n-\ell+1,
\end{array} \right.\\
&&
y_{2n-\ell+1, j} = \ol{y_{j, 2n-\ell+1}} = \left\{ \begin{array}{ll} 
        1 & \mbox{if } j = 2n \\ 
        b & \mbox{if } j = 1, \\ 
        c & \mbox{if } j = \ell, \\ 
        0 & \mbox{if } j \ne 1, \ell, 2n,
\end{array} \right.
\end{eqnarray*}
where $a \in k$ and $b, c \in k'$. Then, by the property $y = y^*$ and $j_{2n}[y] = j_{2n}$, we see $y$ has the required shape, and it is clear the stabilizer $B_y$ contains the elements in the statement.
\end{proof}

\medskip
\begin{lem} \label{lem: upper}
Assume $x \in X_n$, $d_1(x) = 0$ and the first non-zero entry in the $2n$-th column from the bottom stands at $(\ell, 2n)$ with $1 < \ell \leq  n$.
Then there is some $y = (y_{i,j}) \in B\cdot x$ which satisfies
\begin{eqnarray*}
&&
y_{1,j} = y_{j,1} = \delta_{j, 2n-\ell+1}, \quad y_{2n, j} = y_{j, 2n} = \delta_{\ell, j}, \\
&&
y_{\ell, j} = y_{j, \ell} = \delta_{2n, j}, \quad 
y_{2n-\ell+1, j} = y_{j, 2n-\ell+1} = \delta_{1, j},
\end{eqnarray*}
where $\delta_{i,j}$ is the Kronecker delta.

\noindent
The stabilizer $B_y$ contains $t(\bfb)$ with $b_1 = b_\ell = b \in k^\times$ and the remaining $b_i$ being $1$, and the character $\psi_1$ is not trivial on $B_y$.
\end{lem}

\begin{proof}
By the action of $B$, we may assume $x = (x_{i,j})$ satisfies $x_{2n, j} = x_{j, 2n} = \delta_{j, \ell}$.
Then, by the action of $B$, we may take $y \in B\cdot x$ such that
\begin{eqnarray*}
&&
y_{2n, j} = \ol{y_{j, 2n}} = 
\delta_{j,\ell}\\
&&
y_{\ell,j} = \ol{y_{j, \ell}} = \left\{ \begin{array}{ll} 
        b & \mbox{if } j = 1\\ 
        1 & \mbox{if } j = n \\ 
        0 & \mbox{if } j \ne 1, n,
\end{array} \right.
\end{eqnarray*}
where $a, b \in k'$. Then, by the property $y^* = y$ and $j_{2n}[y] = j_{2n}$, we see $y$ has the required shape, and it is clear that the stabilizer $B_y$ contained the elements in the statement.
\end{proof}

\medskip
\begin{lem} \label{lem: corner}
Assume $x \in X_n$, $d_1(x) = 0$ and $x_{2n, i} = 0$ for $2 \leq i \leq 2n$. Then 
$x$ has the following shape:  
$$
x = \left( \begin{array}{c|c|c} 
 {*} & {*} & \xi \\ 
 \hline 
 {*} & y & \begin{array}{c}0\\ \vdots \\ 0 \end{array} \\
 \hline
 \xi & 0 \cdots 0 & 0
 \end{array}\right), \quad \xi = \pm 1, \; y \in \wt{X}_{n-1}, \; \prod_{i=1}^{n-1}\, d_i(y) = 0.
$$
\end{lem}

\begin{proof}
Since $j_{2n}[x]= j_{2n}$, we see ${x_{2n,1}}^2 = 1$ and $x$ has the shape written as above and $y = y^* \in \wt{X}_{n-1}$. 
If $y$ was diagonalizable by the action of $B_{n-1}$, then $\Phi_{xj_{2n}}(t) = (t^2-1)^{n-1}(t-\xi)^2$, which contradicts to the fact $x \in X_{n}$; hence $y$ cannot be diagonalizable and $\prod_{i=1}^{n-1}\, d_i(y) = 0$.
\end{proof}

\medskip
\begin{lem} \label{big corner}
Assume that $x \in X_n$ has the following shape:
$$
x = \left( \begin{array}{c|c|c}
{*} 
& {*} & {x_r}\\
      \hline
{*} & y & {0}\\
\hline
{x_r^*} & {0} & {0}
\end{array}\right), \quad
x_r = \begin{pmatrix} 
      {*} & {} & \xi_1\\ {} & 
\iddots
& {} \\ \xi_r & {} & {0}\end{pmatrix}.
$$
{\rm (i)} Any anti-diagonal entry of $x_r$ equals to $\pm1$, and under the $B_n$-action, we may change $x_r$ into $dj_r$ with $d = Diag(\xi_1, \ldots, \xi_r)$.

\medskip
\noindent
{\rm (ii)} Assume further $r = n$. Then $r$ is even, the numbers of $+1$ and $-1$ within $\set{\xi_i}{1 \leq i \leq n}$ are the same, and $a_{ij} = 0$ if $\xi_i = \xi_j$. 
The stabilizer $B_x$ contains $t(\bfb)$ such that
$$ 
b_i = \left\{ \begin{array}{ll}
b & \mbox{if } \xi_i = 1\\
b^{-1} & \mbox{if } \xi_i = -1
\end{array}\right. \quad (b \in k^\times),
$$
and the character $\psi_1$ is not trivial on $B_x$.
\end{lem}

\begin{proof}
(i) We prove by induction $r$. It is clear for $r = 1$. We assume the assertion holds for $r$ and consider the case $r+1$. Then we may assume that the upper right block $x_{r+1}$ can be written as
$$
x' = \left(\begin{array}{c|ccc}
a_1 & 0 & {} & \xi_1 \\
\vdots & {} & \iddots
& {}\\
a_r & \xi_r & {} & 0\\
\hline
\xi & 0 &\cdots &0
\end{array} \right).
$$
Since $j_{2n}[x] = j_{2n}$, we have $x'j_{r+1}x' = j_{r+1}$ and $a_i = 0$ if $\xi = \xi_i$.
Then setting $b \in B$ as 
$$
b = \left( \begin{array}{c|c|c}
c & 0 & 0 \\
\hline
0 & 1_{2(n-r-1)} & 0\\
\hline
0 & 0 & j_{r+1}c^{*-1}j_{r+1}
\end{array}\right), \qquad c = 
\left(\begin{array}{c|c}
1_r & {\begin{array}{c} -a_1\xi/2\\ \vdots \\ -a_r\xi/2 \end{array}}\\
\hline
0 & 1
\end{array} \right),
$$
the upper right $(r+1)$ by $(r+1)$ block of $b \cdot x$ becomes $Diag(\xi_1, \ldots, \xi_r, \xi)j_{r+1}$ as required.

(ii) Assume $r = n$. Since $\Phi_{xj_{2n}}(t) = (t^2-1)^n$, the numbers of $1$ and $-1$ within $\set{\xi_i}{1 \leq i \leq n}$ are the same.
Since $ad + da = 0_r$, we see $a_{ij} = 0$ if $\xi_i = \xi_j$, and $t(\bfb)$ the above type is contained in $B_x$.
\end{proof}

\bigskip
Now, in order to establish the condition $(A3)$, 
by Lemma~\ref{big corner},  it suffices to consider  $x$ of the following type:

\begin{eqnarray*}
&&
x = \left( \begin{array}{c|c|c} 
{*} & {*} & dj_r \\
\hline
{*} & y & 0\\
\hline
j_rd & 0 & 0
\end{array}\right) \in X_n \backslash X_n^{op}, \quad \begin{array}{l} d = Diag(\xi_1, \ldots, \xi_r), \; \xi_i = \pm 1,\\
y \in \widetilde{X}_m , \; y_{2m,j} \ne 0, \; \mbox{for some $j > 1$}. \end{array}
\end{eqnarray*}
By the action of $B_m$, we may change $y$, into the same shape as in Lemmas~\ref{lem: d1 ne 0}, \ref{lem: lower}, or \ref{lem: upper}, keeping the shape of $x$ as above. Then, $\psi_{r+1}$ is not trivial  on the stabilizer of $B_n$ at the new $x$.



\hfill\qedsymbol
  
\vspace{2cm}

\section{Spherical Fourier transform and Plancherel formula on $\SKX$}
\label{sec:4}

We consider the Schwartz space  
\begin{eqnarray*}
\SKX &=& \set{\vphi: X \longrightarrow \C}{\mbox{left $K$-invariant, compactly supported}},
\end{eqnarray*}
which is spanned by the characteristic function of $K \cdot x, \; x \in X$, and an $\hec$-submodule of $\CKX$ by the convolution product.
We define the modified spherical function 
\begin{eqnarray}
\Psi(x; z) = \omega(x; z) \big{/} \omega(1_{2n};z) \in \calR = \C[q^{\pm z_1}, \ldots, q^{\pm z_n}]^W,
\end{eqnarray}
then by Theorem~\ref{th: explicit}, Corollary~\ref{cor}, (\ref{HL-Q}), we have
\begin{eqnarray} \label{Psi-P}
\Psi(x_\lam; z) &=&
\frac{(1+q^{-1})^n}{w_n(-q^{-1})^2} \cdot c_\lam \cdot Q_\lam(z) \nonumber \\ 
&=&
c_\lam w_\lam P_\lam(z), \quad w_\lam = \frac{\wt{w_\lam}(-q^{-1})}{w_n(-q^{-1})^2}.
\end{eqnarray}
We define the spherical Fourier transform 
\begin{eqnarray} \label{def F-trans}
\begin{array}{lcll}
F :& \SKX &\longrightarrow & \calR\\
{} &  \vphi &\longmapsto &F(\vphi)(z) = \int_{X}\, \vphi(x)\Psi(x; z)dx,
\end{array}
\end{eqnarray}
where $dx$ is a $G$-invariant measure on $X$. 
There is a $G$-invariant measure on $X$, since $X$ is a disjoint union of two $G$-orbits, and $G$ is reductive. We don't need to fix the normalization of $dx$ at this moment, we will determine suitably afterward (cf.~Theorem~\ref{th: Plancherel}). We denote by $v(K\cdot x)$ for the volume of $K\cdot x$ by $dx$.
For $\lam \in \Lam_n^+$, we denote by $\ch_\lam$ the characteristic function of $K\cdot x_\lam$. Then, for $\lam \in \Lam_n^+$
\begin{eqnarray} \label{F(lam)}
F(\ch_\lam)(z) = v(K\cdot x) \Psi(x_\lam; z) = c_\lam w_\lam v(K\cdot x_\lam) P_\lam(z).
\end{eqnarray}
We regard $\calR$ as an $\hec$-module through the Satake isomorphism
$$
\lam_z: \hec \stackrel{\sim}{\longrightarrow} \C[q^{\pm 2z_1}, \ldots, q^{\pm 2z_n}]^W = \calR_0. 
$$

\medskip
\begin{thm} \label{th: sph trans}
The spherical Fourier transform $F$ gives an $\hec$-module isomorphism 
\begin{eqnarray*}
\SKX \stackrel{\sim}{\rightarrow} \C[q^{\pm z_1}, \ldots, q^{\pm z_n}]^W (= \calR),
\end{eqnarray*}
where $\calR$ is regarded as $\hec$-module via $\lam_z$. Especially, $\SKX$ is a free $\hec$-module of rank $2^n$. 
\end{thm}

\begin{proof}
Since the set $\set{\ch_\lam}{\lam \in \Lam_n^+}$ forms a $\C$-basis for $\SKX$ and 
$\set{P_\lam(z)}{\lam \in \Lam_n^+}$ forms a $\C$-basis for $\calR$ (cf. Proposition~\ref{prop:Pbasis}), $F$ is bijective by (\ref{F(lam)}).
Hence $F$ is an $\hec$-module isomorphism, since we have for $f \in \hec$ and $\vphi \in \SKX$, 
\begin{eqnarray*}
F(f * \vphi) &=& \int_{X} \int_{G}f(g)\vphi(g^{-1}\cdot x) \Psi(x; z) dgdx \\
&=&
\int_G \int_X f(g^{-1})\vphi(y)\Psi(g\cdot y; z)dy dg = \int_X \vphi(y) \int_G f(g^{-1})\Psi(g\cdot y; z) dg dy\\
&=&
\lam_z(f) \int_X \vphi(y)\Psi(y; z)dy = \lam_z(f) F(\vphi),
\end{eqnarray*}
where we use the fact $f(g) = f(g^{-1})$ for $g \in G$.
Since we see
$$
\calR = \C[q^{z_1}+q^{-z_1},\ldots, q^{z_n}+q^{-z_n}]^{S_n}, \quad 
\calR_0 = \C[q^{2z_1}+q^{-2z_1},\ldots, q^{2z_n}+q^{-2z_n}]^{S_n}, 
$$
$\calR$ is a free $\calR_0$-module of rank $2^n$, and $\SKX$ is a free $\hec$-module of rank $2^n$.
\end{proof}

\bigskip
\begin{cor} \label{cor:4.2}
All the spherical functions on $X$ are parametrized by eigenvalues\\ $z \in \left( \C/\frac{2\pi\sqrt{-1}}{\log q} \Z \right)^n/W$ through $\hec \longrightarrow \C, \; f \longmapsto \lam_z(f)$.  
The set \\
$\set{\Psi(x; z + u)}{u \in \{0, \pi\sqrt{-1}/\log q \}^n }$ forms a basis of the space of spherical functions on $X$ corresponding to $z$.
\end{cor}

\begin{proof}
The former assertion is clear, since a spherical function $\Psi \in \CKX$ satisfies, by definition
\begin{eqnarray} \label{sph-property}
f * \Psi =  \lam_z(f) \Psi, \quad f \in \hec
\end{eqnarray}
for some $z \in \C^n$, and $\lam_z$ is determined by the class of $z$ in $\left( \C/\frac{2\pi\sqrt{-1}}{\log q} \Z \right)^n/W$. 

The above $\Psi(x; z+u)$ are linearly independent spherical function corresponding to the same $\lam_z$ (cf. (\ref{omega(z_chi)}) ). 
We define a pairing on $\SKX \times \CKX$ by
\begin{eqnarray*}
(\vphi, \Psi) = \int_{X}\vphi(x)\Psi(x)dx, \quad \vphi \in \SKX, \; \Psi \in \CKX,
\end{eqnarray*}
which satisfies 
\begin{equation} \label{hec-action}
(f*\vphi, \Psi) = (\vphi, f* \Psi), \quad (f \in \hec, \; \vphi \in \SKX, \; \Psi \in \CKX).
\end{equation}
Let $\set{\vphi_i}{1 \leq i \leq 2^n}$ be a free $\hec$-basis of $\SKX$. 
Assume that $\Psi_j \in \CKX, \; j \in J$ are spherical functions corresponding to the same $\lam_z$, and set 
$$
{\bf a}_j = (a_{ij})_i \in \C^{2^n}, \quad a_{ij} = \int_{X}\vphi_i(x) \Psi_j(x)dx.
$$ 
Then, by (\ref{hec-action}) and (\ref{sph-property}), it is easy to see that for $c_j \in \C$
$$
\begin{array}{lcl}
\sum_{j\in J} c_j \Psi_j = 0 \; (\mbox{in \;} \CKX) &\Longleftrightarrow& (\vphi_i, \sum_j c_j \Psi_j) = 0, 1 \leq i \leq 2^n\\
{} & \Longleftrightarrow &
\sum_{j \in J}\, c_j {\bf a}_j = {\bf0} \; (\mbox{\; in \; } \C^{2^n}).
\end{array}
$$
Hence, there are at most $2^n$ linearly independent spherical functions, and $\Psi(x; z+u)$'s form a basis.
\end{proof}    

\bigskip
We introduce the following inner product on $\calR$.
Set 
\begin{eqnarray*} 
&&
\fra^* = \left\{ \sqrt{-1}\left( \R/\frac{2\pi}{\log q}\Z \right) \right\}^n,
\end{eqnarray*}
and define a measure $d\mu = d\mu(z)$ on $\fra^*$ by 
\begin{eqnarray} \label{d-mu}
d\mu = \frac{1}{n!2^n} \cdot \frac{w_n(-q^{-1})^2}{(1+q^{-1})^n} \cdot \frac{1}{\abs{c(z)}^2}dz,
\end{eqnarray}
where $c(z)$ is defined in (\ref{c(z)}) 
and $dz$ is the Haar measure on $\fra^*$ with $\int_{\fra^*} dz = 1$.
For $P, Q \in \calR$, we define 
\begin{eqnarray*}
\pair{P}{Q}_\calR = \int_{\fra^*} P(z) \ol{Q(z)} d\mu(z).
\end{eqnarray*}
%
\begin{lem}\label{lem-measure}
For $\lam, \mu \in \Lam_n^+$, one has
$$
\pair{P_\lam}{P_\mu}_\calR = \pair{P_\mu}{P_\lam}_\calR = \delta_{\lam, \mu} w_\lam^{-1}.
$$
\end{lem}

\begin{proof}
Applying Proposition \ref{prop:Pbasis} to our case, Hall-Littlewood polynomials associated with the root system of type $C_n$ with $t_s = -q^{-1}$ and $t_\ell = q^{-1}$, we see
\begin{eqnarray*}
\pair{P_\lam}{P_\mu}_\calR 
& =& 
\delta_{\lam,\mu} \cdot \frac{W_{\bf0}(-q^{-1}, q^{-1})}{W_\lam(-q^{-1}, q^{-1})} 
= \delta_{\lam,\mu} \cdot \frac{w_n(-q^{-1})^2}{\wt{w_\lam}(-q^{-1})}\\ 
&=& \delta_{\lam,\mu} w_\lam^{-1},
\end{eqnarray*}
by (\ref{eq:Wl_sp}) and (\ref{Psi-P}). 
\end{proof}

\begin{lem} \label{lem: ratio of vol}
For $\lam, \mu \in \Lam_n^+$ such that $\abs{\lam} \equiv \abs{\mu} \pmod{2}$, 
\begin{eqnarray*}
\frac{v(K\cdot x_\lam)}{v(K \cdot x_\mu)} = \frac{c_\mu^2 w_\mu}{c_\lam^2 w_\lam}.
\end{eqnarray*}
\end{lem}

\begin{proof}
We understand $\Psi(x; z) \in \CKX = \Hom_\C(\SKX, \C)$ as
\begin{eqnarray} \label{Psi}
\Psi(\;; z) = 
\sum_{\xi \in \Lam_n^+}\, \Psi(x_\xi; z) \ch_\xi.
\end{eqnarray}
For $f \in \hec$, we write
\begin{eqnarray}\label{f-xi}
f* \ch_\xi = \sum_{\nu \in \Lam_n^+}\, a^\xi_\nu(f) \ch_\nu, 
\end{eqnarray}
where almost all $a^\xi_\nu(f) = 0$.
Since $\Psi(x; z)$ is a spherical function associated with $[f \mapsto \wt{f}(z) = \lam_z(f)]$ for $f \in \hec$, we have by (\ref{Psi}) and (\ref{f-xi})
\begin{eqnarray*} \label{f-Psi}
\wt{f}(z) \Psi(x_\lam; z) &=& (f * \Psi(\;; z))(x_\lam) = 
\sum_{\xi \in \Lam_n^+}\, \Psi(x_\xi; z) (f*\ch_\xi)(x_\lam) \nonumber \\
&=&
\sum_{\xi \in \Lam_n^+}\, a^\xi_\lam(f) \Psi(x_\xi; z),
\end{eqnarray*}
and by \eqref{Psi-P} 
\begin{eqnarray}  
c_\lam w_\lam \wt{f}(z) P_\lam(z) = \sum_{\xi \in \Lam_n^+}\, a^\xi_\lam(f) c_\xi w_\xi P_\xi(z).
\end{eqnarray}
Taking the inner product with $P_\mu(z)$, we have by Lemma~\ref{lem-measure},
\begin{eqnarray} \label{lem4-4-(1)} 
c_\lam w_\lam \pair{\wt{f}(z) P_\lam(z)}{P_\mu(z)}_\calR = a^\mu_\lam(f) c_\mu.
\end{eqnarray}
By the compatibility of $F$ with $\hec$-action and (\ref{F(lam)}), taking the image of $F$ of (\ref{f-xi}) for $\xi = \mu$, we have
\begin{eqnarray*}
c_\mu w_\mu v(K \cdot x_\mu) \wt{f}(z) P_\mu(z) = \sum_{\nu \in \Lam_n^+}\, a^\mu_\nu(f) c_\nu w_\nu v(K \cdot x_\nu) P_\nu(z).
\end{eqnarray*}
Taking the inner product with $P_\lam(z)$, we have by Lemma~\ref{lem-measure},
\begin{eqnarray} \label{lem4-4-(2)}
c_\mu w_\mu v(K \cdot x_\mu) \pair{\wt{f}(z) P_\mu(z)}{P_\lam(z)}_\calR = a^\mu_\lam(f) c_\lam v(K \cdot x_\lam).
\end{eqnarray}
By (\ref{lem4-4-(1)}) and (\ref{lem4-4-(2)}), we have for any $f \in \hec$

\begin{equation}  \label{lem4-4-(3)}
c_\mu^2 w_\mu v(K \cdot x_\mu) \pair{\wt{f}(z) P_\mu(z)}{P_\lam(z)}_\calR = 
c_\lam^2 w_\lam v(K \cdot x_\lam)\pair{\wt{f}(z) P_\lam(z)}{P_\mu(z)}_\calR. 
\end{equation}
Now assume $\abs{\lam} \equiv \abs{\mu} \pmod{(2)}$, then $x_\lam = g \cdot x_\mu$ for some $g \in G$.
Taking $f \in \hec$ as the characteristic function of the double coset $KgK$, we see $a^\mu_\lam(f) > 0$ and 
\begin{eqnarray*}
0 \ne \pair{\wt{f}(z)P_\lam(z)}{P_\mu(z)}_\calR =
\pair{P_\lam(z)}{\wt{f}(z) P_\mu(z)}_\calR
= \pair{\wt{f}(z)P_\mu(z)}{P_\lam(z)}_\calR.
\end{eqnarray*}
Thus we obtain, by (\ref{lem4-4-(3)})
\begin{eqnarray*}
c_\mu^2 w_\mu v(K \cdot x_\mu)  =
c_\lam^2 w_\lam v(K \cdot x_\lam),
\end{eqnarray*}
which completes the proof.
\end{proof}

\bigskip
We recall that $X$ decomposed into two $G$-orbits: 
$$
X = G \cdot x_0 \sqcup G \cdot x_1, \qquad x_0 = 1_{2n}, \; x_1 = Diag(\pi, 1, \ldots, 1, \pi^{-1}).
$$

\begin{thm}[Plancherel formula on $\SKX$] \label{th: Plancherel} 
Let $d\mu$ be the measure defined by (\ref{d-mu}). By the normalization of $G$-invariant measure $dx$ such that 
\begin{eqnarray} \label{vol K-lam}
v(K\cdot x_\lam) = c_\lam^{-2}w_\lam^{-1}, \qquad \lam \in \Lam_n^+,
\end{eqnarray}
one has for any $\vphi, \psi \in \SKX$
\begin{eqnarray} \label{the formula}
\int_{X} \vphi(x)\ol{\psi(x)} dx = \int_{\fra^*}F(\vphi)(z) \ol{F(\psi)(z)} d\mu(z).
\end{eqnarray}
\end{thm}

\begin{proof}
We may normalize $dx$ on $X$ according to the $G$-orbits as   
$v(K\cdot x_0) = 1$ on $G \cdot x_0$ and 
$$
v(K\cdot x_1) = q^{2n-1}\frac{(1-(-q^{-1})^n)^2}{1+q^{-1}} \; \mbox{ on } \; G \cdot x_1,
$$
then it satisfies (\ref{vol K-lam}) by Lemma~\ref{lem: ratio of vol}.
It suffices to show the identity (\ref{the formula}) for $\ch_\lam$ and $\ch_\mu$ ($\lam, \mu \in \Lam_n^+$). Indeed,
\begin{eqnarray*}
\int_{X}\, \ch_\lam(x) \ol{\ch_\mu(x)} dx = \delta_{\lam, \mu}\, v(K\cdot x_\lam) = \delta_{\lam, \mu} c_\lam^{-2} w_\lam^{-1},
\end{eqnarray*}
while
\begin{eqnarray*}
\lefteqn{\int_{\fra^*}F(\ch_\lam)(z) \ol{F(\ch_\mu)(z)} d\mu(z)}\\
 &=&
c_\lam w_\lam v(K\cdot x_\lam) \cdot c_\mu w_\mu v(K \cdot x_\mu) \pair{P_\lam(z)}{P_\mu(z)}_\calR
= \delta_{\lam, \mu} c_\lam^{-2} w_\lam^{-1},
\end{eqnarray*}
which completes the proof.
\end{proof}

\begin{cor}[Inversion formula]
For any $\vphi \in \SKX$, 
$$
\vphi(x) = \int_{\fra^*} F(\vphi)(z) \Psi(x; z) d\mu(z), \quad x \in X.
$$
\end{cor}

\begin{proof}
For any $x \in X$, 
setting $\psi_x$ the characteristic function of $K\cdot x$,
we have
\begin{eqnarray*}
\vphi(x) &=& \frac{1}{v(K\cdot x)} \int_{X} \vphi(y) \psi_x(y) dy \\
&=&
\frac{1}{v(K\cdot x)} \int_{\fra^*} F(\vphi)(z) \ol{F(\psi_x)(z)} d\mu(z)\\
&=&
\int_{\fra^*} F(\vphi)(z) \Psi(x;z) d\mu(z). 
\end{eqnarray*}
\end{proof}

\vspace{2cm}



\appendix

\section{The spaces of unitary hermitian matrices}
\label{sec:appA}

In this subsection, let $k'$ be a quadratic extension of a field $k$ of characteristic $0$, and consider hermitian matrices with respect to $k'/k$. As usual we denote by $A^* \in M_{nm}(k')$ the conjugate transpose of $A \in M_{mn}(A)$. 

We fix a natural number $m$ and set $\calH_m(k') = \set{x \in GL_m(k')}{x^* = x}$. We define the unitary group for $x \in \calH_m(k')$ by
\begin{equation} \label{unitary gr}
U(x) = \set{g \in GL_m(k')}{x[g] = x}, \quad \mbox{where }\; x[g] = g^*xg = g^* \cdot x,
\end{equation}
and, in particular
\begin{eqnarray*}
G = U(j_m) 
 \quad \mbox{for }\; 
j_m = \begin{pmatrix}
0 & {} & 1\\ {} & \iddots
& {}\\1 & {} & 0
\end{pmatrix} \in M_m.
\end{eqnarray*}
We introduce the space $\wt{X}$ of unitary hermitian matrices and the $G$-action as follows:
\begin{eqnarray}
&&
\wt{X} = \set{x \in G}{x^* = x} = \set{x \in GL_m(k')}{x = x^*, \; (x j_m)^2 = 1_m}, \nonumber \\
&&
\label{the action}
g \cdot x = gxg^* = gxj_mg^{-1}j_m \quad (g \in G, \; x \in \wt{X}).
\end{eqnarray}
We consider these objects as the sets of $k$-rational points of algebraic sets defined over $k$. 
We denote by $\ol{k}$ the algebraic closure of $k$ and realize $G(\ol{k})$ in $\wt{\G} = Res_{k'/k}(GL_m)$ as follows. 
We understand  $\wt{\G} = GL_m(\ol{k}) \times GL_m(\ol{k})$ with the Galois group $\Gamma = Gal(\ol{k}/k)$ action defined by 
\begin{eqnarray*}
\sigma(g_1, g_2) = \left\{ \begin{array}{ll}
({g_1}^\sigma, {g_2}^\sigma) & \mbox{if } \sigma\vert_{k'} = id \\[2mm]
({g_2}^\sigma, {g_1}^\sigma) & \mbox{if } \sigma\vert_{k'} = \tau, 
\end{array} \right. \quad (\sigma \in \Gamma, \; (g_1, g_2) \in \wt{\G}),
\end{eqnarray*}
where $g^\sigma = ({g_{ij}}^\sigma)$ for $g = (g_{ij}) \in GL_m(\ol{k})$ and $\langle{\tau}\rangle = Gal(k'/k)$.
Then 
\begin{eqnarray*}
&&
\wt{\G}(k) = \set{(g, g^\tau)}{g \in GL_m(k')},\\
&&
G(\ol{k}) = \set{(g_1, g_2) \in \wt{\G} }{{}^tg_2j_m g_1 = j_m} 
= \set{(g, j_m{}^tg^{-1} j_m)}{g \in GL_m(\ol{k})},
\end{eqnarray*}
and the involution $[g \longmapsto g^*]$ on $G$ can be extended as $(g_1, g_2)^* = ({}^tg_2, {}^tg_1)$ on $G(\ol{k})$. Hence we see
\begin{eqnarray*}
\wt{X}(\ol{k}) 
&=&
\set{(x, j_m{}^tx^{-1}j_m)}{x = j_mx^{-1}j_m \in GL_m(\ol{k})} \nonumber \\
&=&
\set{(x, j_m{}^tx^{-1}j_m)}{x \in GL_m(\ol{k}), \;(xj_m)^2 = 1_m},
\end{eqnarray*}
and the action of $G(\ol{k})$ on $\wt{X}(\ol{k})$ can be written as 
\begin{eqnarray*}
(g, j_m{}^tg^{-1} j_m)) \star (x, j_m{}^tx^{-1}j_m) &=&
(g, j_m{}^tg^{-1} j_m) (x, j_m{}^tx^{-1}j_m) (j_mg^{-1}j_m, {}^tg) \nonumber \\
&=&
(gxj_mg^{-1}j_m, j_m{}^tg^{-1}{}^tx^{-1}j_m{}^tg) \nonumber \\
&=& 
(gxj_mg^{-1}j_m, j_m{}^t(gxj_mg^{-1}j_m)^{-1}j_m).
\end{eqnarray*}
Hence we may identify
\begin{eqnarray} 
&&
G(\ol{k}) = GL_m(\ol{k}), \quad \wt{X}(\ol{k}) = \set{x \in G(\ol{k})}{(xj_m)^2 = 1_m}, \nonumber\\ 
&& \label{star-action}
g \star x = gxj_mg^{-1}j_m, \quad (g \in G(\ol{k}), \; x \in \wt{X}(\ol{k})),
\end{eqnarray}
where  
$g \star x = g\cdot x$ if $g \in G$ and $x \in X$ (cf.~(\ref{the action}) ).

\begin{prop} \label{prop A-1}
The space $\wt{X}(\ol{k})$ decomposes into $G(\ol{k})$-orbits as follows: 
$$
\wt{X}(\ol{k}) = \bigsqcup_{i=0}^{m} \set{x \in \wt{X}(\ol{k})}{\Phi_{xj_m}(t) = (t-1)^i(t+1)^{m-i}},
$$
where $\Phi_y(t)$ is the characteristic polynomial of the matrix $y$.
\end{prop}

\begin{proof}
We consider the following $G(\ol{k})$-set:
\begin{eqnarray*}
&&
Y(\ol{k}) = \set{y \in G(\ol{k})}{y^2 = 1}, \\
&&
g \circ y = gyg^{-1} \qquad(g \in G(\ol{k}), \; y \in Y(\ol{k})).
\end{eqnarray*}
Then the $G(\ol{k})$-orbits in $Y(\ol{k})$ are determined by characteristic polynomials as follows
\begin{eqnarray*} \label{corres-orbit}
Y(\ol{k}) = 
\bigsqcup_{i=0}^{m}\, \set{y \in Y(\ol{k})}{\Phi_y(t) = (t-1)^i(t+1)^{m-i}}. 
\end{eqnarray*}
Since the map  
$$
\psi: \wt{X}(\ol{k}) \longrightarrow Y(\ol{k}), \quad x \longmapsto xj_m
$$
is bijective and $G(\ol{k})$-equivariant, we have the $G(\ol{k})$-orbit decomposition for $\wt{X}(\ol{k})$ as required.
\end{proof}

\bigskip
We take the $G(\ol{k})$-orbit in $\wt{X}(\ol{k})$ containing $1_m$ and set
\begin{eqnarray*} \label{def of X}
X(\ol{k}) = G(\ol{k}) \star 1_m, \qquad X = X(\ol{k}) \cap G = X(\ol{k}) \cap \wt{X}.
\end{eqnarray*}
By Proposition~\ref{prop A-1} and its proof, we see
\begin{eqnarray} \label{s1-X}
X = \set{x \in \wt{X}}{\Phi_{xj_m}(t) = \Phi_{j_m}(t)}.
\end{eqnarray}


\section{Root systems and Hall-Littlewood polynomials}
\label{sec:appB}
In this appendix, we summarize several properties of symmetric polynomials,  in particular Hall-Littlewood polynomials, which
are understood as special cases of Macdonald polynomials.
 In Appendix \ref{sec:appB}, 
we use $\N_0=\N\cup\{0\}$.

\subsection{Root systems}
Let $V$ be an $n$-dimensional real vector space equipped with inner
product $\langle~,~\rangle$.  Let $\Sigma$ be an irreducible
reduced root system of rank $n$ in $V$.  We fix a set
$\Sigma_0=\set{\alpha_i}{1 \leq i \leq n}$ of simple roots, and denote by $\Sigma^+$ 
the set of all positive roots of $\Sigma$ with respect to $\Sigma_0$.
We denote by $\Lambda_0=\set{\mu_i}{1 \leq i \leq n}$  the set of fundamental weights satisfying
$$
\langle\alpha_i^\vee,\mu_j\rangle=\delta_{i,j}, \quad  \alpha^\vee=2\alpha/\langle\alpha,\alpha\rangle \in V,
$$
and set 
\begin{eqnarray*}
&&
\mbox{the weight lattice: } \Lambda =\bigoplus_{i=1}^n\mathbb{Z}\mu_i,\\
&&
\mbox{the set of all dominant weights: } \Lam^+ = \bigoplus_{i=1}^n\N_0\mu_i,\\
&&
\mbox{the coroot lattice: } Q^\vee=\bigoplus_{i=1}^n\mathbb{Z}\alpha_i^\vee.
\end{eqnarray*}
Let $\tau_\alpha$ be the reflection defined by
$\tau_\alpha(v)=v-\langle v,\alpha^\vee\rangle\alpha$ for $v\in V$, and
$W=W(\Sigma)$ be the Weyl group generated by
$\set{\tau_\alpha}{\alpha\in\Sigma}$.

\subsection{Symmetric polynomials}
We denote by $\mathbb{C}[\Lam]$ the group algebra of the lattice
$\Lambda$, which is spanned by the formal exponentials $e^\lambda$
with $\lambda\in \Lam$.  The Weyl group acts on $\C[\Lam]$ via the
action on $\Lam$: $\sigma(e^\alp) = e^{\sigma(\alp)}$ for $\sigma \in
W$ and $\lam \in \Lam$. We denote by $\mathbb{C}[\Lam]^W$ the
subalgebra of $W$-invariant elements in $\mathbb{C}[\Lam]$.  For simplicity of notation we
set
\begin{eqnarray} \label{calR and calR0}
\mathcal{R}=\mathbb{C}[\Lambda]^W, \quad  \mathcal{R}_0=\mathbb{C}[2\Lambda]^W = \calR \cap \C[2\Lam].
\end{eqnarray}
For each $\lam \in \Lam^+$, we set  
\begin{equation*}
  m_{\lambda}=\sum_{\nu\in W\lambda}e^\nu, \qquad W\lambda=\{\sigma\lambda\in\Lambda~|~\sigma\in
W\},
\end{equation*}
then it is easy to see that $\set{m_\lambda}{\lambda\in \Lambda^+}$ forms a $\C$-basis for $\mathcal{R}$.
We define a partial order $\prec$ on $\Lam$ by
\begin{eqnarray} \label{partial order}
\mu\prec\lambda \; \mbox{ if and only if }\; \lambda-\mu \in Q^+, 
\end{eqnarray}
where
$$
Q^+ = \bigoplus_{i=1}^n\N_0\alp_i \; \subset \; \Lam.
$$

\begin{lem} 
  \label{lm:mono}
  For $\lambda,\mu\in \Lambda^+$, there exist some $a_\nu \in \N_{0} \; (\nu \in \Lam^+)$ such that
  \begin{equation*}
    m_{\lambda}m_{\mu}=m_{\lambda+\mu}+\sum_{\substack{\nu\in \Lambda^+\\\nu\precneqq\lambda+\mu}} a_\nu m_{\nu}.
  \end{equation*}
\end{lem}
\begin{proof} 
Let $A=W\lambda+W\mu$ and $\kappa\in A$.
The $W$-invariance of
$A$ implies that there exists $\sigma\in W$ such
that $\sigma\kappa\in A\cap \Lambda^+$. 
Write $\nu=\sigma\kappa=\sigma_1\lambda+\sigma_2\mu$ for $\sigma_1,\sigma_2\in W$.
Since 
$\lambda\succ \sigma_1\lambda$ and $\mu\succ\sigma_2\mu$,
we have 
$\lambda+\mu\succ\sigma_1\lambda+\sigma_2\mu=\nu$.
Hence we obtain the result.
\end{proof}

\medskip
\begin{prop}
  \label{prop:basis}
The algebra $\mathcal{R}$ is a free $\mathcal{R}_0$-module of rank $2^n$. Any family $\set{p_\mu}{\mu \in M}$ satisfying
\begin{eqnarray}
\label{triang}
&&
 p_\mu=m_\mu+\sum_{\substack{\nu\in \Lambda^+\\\nu\precneqq\mu}} b_{\mu \nu} \, m_{\nu}, \quad (b_{\mu\nu} \in \C),
\\ 
&&
M=\set{\sum_{i=1}^n c_i \mu_i}{c_i\in\{0,1\}, \; 1 \leq i \leq n}
\notag
\end{eqnarray}
forms a free $\calR_0$-basis for $\calR$.
\end{prop}

\begin{proof}
Since $\Lam^+$ is a disjoint union of $\mu + 2\Lam^+, \; (\mu \in M)$,
we see $\set{p_\mu}{\mu \in M}$ forms a $\calR_0$-basis for $\calR$, by Lemma~\ref{lm:mono}.
\end{proof}

\subsection{Hall-Littlewood polynomials}
We introduce a family of orthogonal polynomials, which is a special
case of Macdonald polynomials, or can be regarded as Hall-Littlewood
polynomials associated with root systems.

We fix parameters $t_\alpha\in\mathbb{R}$ with $|t_\alpha|<1$ for each
$\alpha\in\Sigma$ such that $t_\alpha=t_\beta$ if $(\alp, \alp)=(\beta, \beta)$,
hence there are at most two independent parameters among $t_\alpha$'s.
We define Hall-Littlewood polynomial associated with the root system $\Sigma$, for $\lambda\in\Lam^+$, 
\begin{equation}
\label{eq:def_HL}
  P_\lambda=\frac{1}{W_\lambda(t)}\sum_{\sigma\in W}\sigma\Bigl(e^\lambda\prod_{\alpha\in\Sigma^+}\frac{1-t_\alpha e^{-\alpha}}{1-e^{-\alpha}}\Bigr), 
\end{equation}
where $W_\lambda$ is the stabilizer in $W$ at $\lambda$, and the Poincar\'e polynomial of a subgroup $W'$ of $W$ is defined by
\begin{equation*}
  W'(t)=\sum_{\sigma\in W'}\prod_{\alpha\in \Sigma^+\cap (-\sigma\Sigma^+)}t_\alpha .
\end{equation*}

We define the measure $d\mu = d\mu(z)$ on $\fra^*$ as follows.
\begin{align*} 
  \fra^* &= \sqrt{-1}\Bigl( V/\frac{2\pi}{\log q}Q^\vee \Bigr), \qquad \int_{\fra^*}dz = 1, \nonumber\\
\label{measure}
  d\mu &= d\mu(z) = \frac{W(t)}{\sharp W} \cdot \frac{dz}{\abs{c(z)}^2},\\
  c(z) &= \prod_{\alp \in \Sigma^+} \frac{1 -t_\alpha q^{\pair{\alp}{z}}}{1-q^{\pair{\alp}{z}}}.\nonumber
\end{align*}
We regard $e^\lambda\in\mathbb{C}[\Lam]$ as a function on
$\fra^*$ via $e^\lambda(z)=q^{-\langle \lambda,z\rangle}$.  Then 
we have
\begin{equation*}
  P_\lambda(z)=\frac{1}{W_\lambda(t)}\sum_{\sigma\in W}\sigma\Bigl(q^{-\langle\lambda,z\rangle}
c(z)\Bigr),
\end{equation*}
and
an
inner product on $\calR$ is defined by
\begin{equation*}
  \label{eq:def_inner}
  \langle P, Q\rangle_\calR = \int_{\fra^*} P(z)\ol{Q(z)} d\mu(z)
\end{equation*}
for $P,Q\in\calR$. Then it is known that the following holds.

\medskip
\begin{prop}[{\cite[\S3,\S10]{Mac}}]
  \label{prop:Pbasis}
The subset $\{P_\lambda\}_{\lambda\in \Lam^+}$ of $\calR$ satisfies the
  triangularity condition of type \eqref{triang} and forms an orthogonal
  basis of $\calR$ with respect to the inner product $\langle\ ,\ \rangle_\calR$
  with
  \begin{equation*}
    \langle P_\lambda, P_\mu\rangle_\calR =\delta_{\lambda,\mu} 
\frac{W(t)}{W_{\lambda}(t)}.
  \end{equation*}
\end{prop}

By definition, we see $W(t)=W_{\mathbf{0}}(t)$. In the case of the root system of type $C_n$,
the explicit forms of the Poincar\'e polynomial $W_{\lambda}(t)$ and 
the Hall-Littlewood polynomial $P_\lambda$
with the special parameters used in the paper
are given
in \eqref{eq:Wl_sp} 
and
in \eqref{eq:HL_sp} respectively.

\subsection{Poincar\'e polynomials of stabilizers 
and Hall-Littlewood polynomials in the root system of type $C_n$}

\begin{prop}
  \label{prop:poincare}
  For $\lambda\in\Lam^+$, let
$$
\Sigma_{0,\lam} = \{\alpha\in\Sigma_0~|~\langle\alpha,\lambda\rangle=0\}=\bigsqcup_k
  \Phi_k,
$$ 
where each $\Phi_k$ is a connected component of the Dynkin
  diagram corresponding to $\Sigma_{0, \lam}$.  Denote by $W_k$  the Weyl group
  generated by the reflections for $\alpha\in \Phi_k$.  Then
  $W_\lambda(t)=\prod_{k}W_k(t)$.
\end{prop}
\begin{proof}
Since the stabilizer $W_\lambda$ is generated by the reflections
that fix $\lambda$ (cf.~\cite[1.12, Theorem]{Hum}), we see that
$W_\lambda$ is generated by $\{\sigma_\alpha\}_{\alpha\in \Sigma_\lambda}$, where
\begin{equation*}
  \Sigma_\lambda=\{\alpha\in\Sigma~|~\langle\alpha,\lambda\rangle=0\}.
\end{equation*}
Since $\Sigma_\lambda$ is a root system, it is sufficient to show that 
$\Sigma_{0,\lambda}$ is a fundamental system of $\Sigma_\lambda$.
Write $\alpha=\sum_{i=1}^n c_i\alpha_i\in\Sigma_\lambda$.
Then we have all $c_i\geq0$ or all $c_i\leq0$. Furthermore
\begin{equation*}
  0=\langle\alpha,\lambda\rangle=\sum_{i=1}^n c_i\langle\alpha_i,\lambda\rangle.
\end{equation*}
Since
$\langle\alpha_i,\lambda\rangle\geq 0$ for $1\leq i\leq n$,
we have
$c_i=0$ for such $i$ that
$\langle\alpha_i,\lambda\rangle\neq 0$.
Thus 
\begin{equation*}
  \alpha=\sum_{\alpha_i\in\Sigma_{0,\lambda}} c_i\alpha_i.
\end{equation*}
\end{proof}

We use the realization of the root system of $C_n$ introduced in \eqref{short and long}.
Then we have
\begin{align*}
  \Sigma_0 &= \set{e_i - e_{i+1}}{1 \leq i \leq n-1}\cup\{2e_n\},\\
  \Lambda_0 &= \set{\mu_i=\sum_{k=1}^i e_k}{1 \leq i \leq n},
\end{align*}
and  we can identify the set $\Lam^+$ of dominant weights with 
$$\Lam_n^+ = \set{\lam \in \Z^n}{\lam_1 \geq \cdots \geq \lam_n \geq 0}.$$
 We fix any $\lam \in \Lam_n^+$ and understand 
\begin{eqnarray*}
\lam = \sum_{i=1}^n \lam_ie_i = \sum_{i=1}^{n-1}\, (\lam_i-\lam_{i+1})\mu_i + \lam_n\mu_n \in \Lam^+ .
\end{eqnarray*}
Then we have the decomposition
\begin{equation*}
  \Sigma_{0,\lam} = \set{\alpha\in\Sigma_0}{\langle\alpha,\lambda\rangle=0}=\bigsqcup_{k=0}^\infty
  \Phi_k,
\end{equation*}
where
\begin{equation*}
  \Phi_k=
  \begin{cases}
    \{\alpha_i\in\Sigma_0~|~\lambda_i=\lambda_{i+1}=k\} \qquad &(k>0)\\
    \{\alpha_i\in\Sigma_0~|~\lambda_i=0\} \qquad &(k=0).
  \end{cases}
\end{equation*}
Let $n_k=\sharp \{i~|~\lambda_i=k\}$.  Note that $n_k=\sharp \Phi_k+1$
for $k>0$ and $n_0=\sharp \Phi_0$.  Then we see that $W_k$ is
isomorphic to the Weyl group of type $A_{n_k-1}$ if $k>0$,
and that $W_0$ is of type $C_{n_0}$. 
Setting $W_{A_0}(t_s)=W_{C_0}(t_s,t_\ell)=1$ formally, we obtain by Proposition~\ref{prop:poincare}
\begin{equation*}
\label{eq:Wl_gen}
  W_\lambda(t)=\Bigl(\prod_{k=1}^\infty W_{A_{n_k-1}}(t_s)\Bigr)\cdot W_{C_{n_0}}(t_s,t_\ell),
\end{equation*}
where $t_s$ and $t_\ell$ are attached to short roots and long roots respectively.

The Poincar\'e polynomials of type $A_{n-1}$ and of type $C_n$ are
respectively given as
\begin{equation*}
\label{eq:Wl_gen1}
  \begin{split}
    W_{A_{n-1}}(t)&=\prod_{i=1}^{n-1}\frac{1-t^{i+1}}{1-t},\\
    W_{C_n}(t_s,t_\ell)&=\prod_{i=0}^{n-1}(1+t_s^it_\ell)\frac{1-t_s^{i+1}}{1-t_s}.
  \end{split}
\end{equation*}
In the case $t_s=-q^{-1}$ and $t_\ell=q^{-1}$, we obtain
\begin{equation*}
  \begin{split}
    W_{A_{n-1}}(-q^{-1})&=\frac{w_n(-q^{-1})}{(1+q^{-1})^n},\\
    W_{C_n}(-q^{-1},q^{-1})&=\frac{w_n(-q^{-1})^2}{(1+q^{-1})^n}.
  \end{split}
\end{equation*}
Finally we arrive at the explicit form of the Poincar\'e polynomials 
and the Hall-Littlewood polynomials in the root system of type $C_n$
with $t_s=-q^{-1}$ and $t_\ell=q^{-1}$ as follows.
\begin{gather}
  \label{eq:Wl_sp}
  W_\lambda(-q^{-1},q^{-1})=\frac{\wt{w_\lam}(-q^{-1})}{(1+q^{-1})^n},\\
  \label{eq:HL_sp}
  \begin{split}
    P_\lambda(z) &= 
    \frac{(1+q^{-1})^n}{\wt{w_\lam}(-q^{-1})}\cdot
    \sum_{\sigma \in W}\, \sigma\left(  q^{-\pair{\lam}{z}} c(z) \right), \\
    c(z) &= \prod_{\alp \in \Sigma_s^+}\, \frac{1 + q^{\pair{\alp}{z}-1}}{1 - q^{\pair{\alp}{z}}}
    \prod_{\alp \in \Sigma_\ell^+}\, \frac{1 - q^{\pair{\alp}{z}-1}}{1 - q^{\pair{\alp}{z}}}. 
  \end{split}
\end{gather}

\section{Relation with the space $X_T$}
\label{sec:appC}

We assume that $k'/k$ be an unramified quadratic extension of $\frp$-adic fields. In this subsection we explain the relevance of the present space $X = X_n$ to the spaces introduced in \cite{Oda}, and show the expectation there is correct for odd residual characteristic case. 

In \cite{Oda}, we have considered for each $T \in \calH_n(k')$
\begin{eqnarray*}
X_T = \frX_T/U(T), \qquad \frX_T = \set{x \in M_{2n,n}(k')}{H_n[x] = T}, \quad H_n = \twomatrix{0}{1_n}{1_n}{0},
\end{eqnarray*}
where 
$
U(H_n) = \set{g \in GL_{2n}(k')}{H_n[g] = H_n} 
$
acts homogeneously on $X_T$ by the left multiplication, and the stabilizer at a point in $X_T$ is isomorphic to $U(T) \times U(T)$ (cf.~\cite[Lemma~1.1]{Oda}).  
We assume $T$ is diagonal and realize $X_T$ as a set of $k$-rational points in an algebraic set defined over $k$. 
We set
\begin{eqnarray*}
\frX_T(\ol{k}) &=& \set{(x,y) \in M_{2n, n}(\ol{k}) \oplus M_{2n, n}(\ol{k})}{{}^tyH_n x = T},
\end{eqnarray*}
on which $\Gamma = Gal(\ol{k}/k)$ acts by 
$$
\sigma(x, y) = \left\{\begin{array}{ll}
(x^\sigma, y^\sigma) & \mbox{if } \sigma \vert_{k'} = id\\
(y^\sigma, x^\sigma) & \mbox{if } \sigma \vert_{k'} = \tau,
\end{array}
\right.
$$
where $\langle\tau\rangle = Gal(k'/k)$.
We set
\begin{eqnarray*}
&&
\U(H_n) = U(H_n)(\ol{k}) = \set{(g_1, g_2) \in GL_{2n}(\ol{k})\times GL_{2n}(\ol{k})}{{}^tg_2 H_n g_1 = H_n},\\
&&
\U(T) = U(T)(\ol{k}) = \set{(h_1, h_2) \in GL_{n}(\ol{k}) \times GL_n(\ol{k})}{{}^th_2 T h_1 = T},\\
&&
\X_T(\ol{k}) = \frX_T(\ol{k})/\U(T) \; \supset \; \X_T(k) = {\X_T}^\Gamma,
\end{eqnarray*}
where we can consider the similar $\Gamma$-action on $\U(H_n)$ and $\U(T)$, since $H_n$ and $T$ are $\Gamma$-invariant.
We identify 
$$
\begin{array}{ll}
\U(H_n)^\Gamma =  \set{(g, \ol{g}) \in GL_{2n}(k') \times GL_{2n}(k')}{{}^t\ol{g}H_ng = H_n} & \mbox{ with } U(H_n),\\[2mm]
\U(T)^\Gamma =  \set{(h, \ol{h}) \in GL_{n}(k') \times GL_{n}(k')}{{}^t\ol{h}Th = T} & \mbox{ with } U(T),
\end{array}
$$
where and henceforth we write $\ol{g}$ instead of $g^\tau$ for a matrix $g$ with entries in $k'$.

\begin{lem}
The map
$$
\vphi_T: X_T \longrightarrow \X_T(k), \; xU(T) \longmapsto (x, \ol{x})\U(T)
$$
is injective.
\end{lem}

\begin{proof}
Assume $\vphi_T(x U(T)) = \vphi_T(y U(T))$. Then, for some $(h_1, h_2) \in \U(T)$, we have $y = xh_1, \; \ol{y} = xh_2$. Taking $n$ linearly independent rows from $x$, we make $x_0 \in GL_n(k')$. Then 
$y_0 = x_0h_1 \in GL_n(k'), \; h_1 = x_0^{-1}y_0 \in GL_n(k')$ and $h_2 = \ol{x_0}^{-1} \ol{y_0} = \ol{h_1} \in GL_n(k')$. Hence $(h_1, h_2) \in U(T)$ and $xU(T) = yU(T)$.
\end{proof}

\bigskip
Hereafter we understand $X_T$ as a subspace of $\X_T(k)$ through $\vphi_T$. Set 
$$
T_1 = \twomatrix{\pi}{0}{0}{1_{n-1}}, \quad \wt{\eta_\pi} = (\eta_\pi, \eta_\pi), \quad \eta_\pi = \twomatrix{\sqrt{\pi}^{-1}}{0}{0}{1_{n-1}}.
$$

\begin{lem} \label{map f}
The map
$$
f : \X_{T_1}(\ol{k}) \longrightarrow \X_{1_n}(\ol{k}), \; (x,y)\U(T_1) \longmapsto (x\eta_\pi, y\eta_\pi)\U(1_n)
$$
is well defined and sends $\X_{T_1}(k)$ into $\X_{1_n}(k)$.
\end{lem}

\begin{proof}
For any $(x,y) \in \frX_{T_1}(\ol{k})$ and $\wt{h} = (h_1, h_2) \in \U(T_1)$, we have
\begin{eqnarray*}
(xh_1\eta_\pi, yh_2\eta_\pi) = (x\eta_\pi, y\eta_\pi) \wt{\eta_\pi}^{-1}\wt{h}\wt{\eta_\pi} \in 
(x\eta_\pi, y\eta_\pi)\U(1_n),
\end{eqnarray*}
hence $f$ is well defined.
Take any $\alp = (x,y)\U(T_1) \in \X_{T_1}(k)$. For each $\sigma \in \Gamma$, there exists some $h_\sigma \in \U(T_1)$ satisfying $\sigma(x,y) = (x, y)h_\sigma$, and $\sigma(\sqrt{\pi}) = \pm\sqrt{\pi}$. Hence
\begin{eqnarray*}
\sigma (f(\alp)) &=& (x, y)h_\sigma \wt{\eta_\pi} (\twomatrix{\pm 1}{0}{0}{1_{n-1}},\twomatrix{\pm 1}{0}{0}{1_{n-1}}) \U(1_n) \\
&=& (x\eta_\pi, y\eta_\pi) \wt{\eta_\pi}^{-1} h_\sigma \wt{\eta_\pi} \U(1_n) 
= (x\eta_\pi, y\eta_\pi) \U(1_n) = f(\alp).
\end{eqnarray*}
Hence $\sigma f = f$ for any $\sigma \in \Gamma$, and $f$ sends the set $\X_{T_1}(k)$ of $k$-rational points into to the set   $\X_{1_n}(k)$.
\end{proof}

\begin{lem}  \label{Hn-orbit}
The set $\X_{1_n}(k)$ contains at least two $U(H_n)$-orbits, $X_{1_n}$ and $f(X_{T_1})$. 
\end{lem}

\begin{proof}
By Lemma~\ref{map f}, we see 
\begin{eqnarray} \label{comm-diag}
\begin{array}{cccc}
X_{T_1}  \subset & \X_{T_1}(k) &\subset & \X_{T_1}(\ol{k}) \\[2mm]
{} & \mapdownr{f} & {} & \mapdownr{f}\\[2mm]
X_{1_n}  \subset & \X_{1_n}(k) &\subset &\X_{1_n}(\ol{k}),
\end{array}
\end{eqnarray}
hence it suffices to show $X_{1_n} \cap f(X_{T_1}) = \emptyset$.
Assume there exists some $(x,\ol{x}) \in \frX_{T_1}, \; (y,\ol{y}) \in \frX_{1_n}, \; (h_1, h_2) \in U(1_n)$ such that $(x\eta_\pi,\ol{x}\eta_\pi) = (y, \ol{y})(h_1, h_2)$. 
Taking suitable linearly independent rows from $x$, we have $x_0 \in GL_n(k')$ and $y_0 = x_0\eta_\pi h_1^{-1} \in GL_n(k')$. Then we see
$$
h_1\eta_\pi^{-1} = y_0^{-1}x_0 \in GL_n(k'), \quad \ol{h_1 \eta_\pi^{-1}} = h_2 \eta_\pi^{-1}.
$$
Setting $h = h_1\eta_\pi^{-1}$, we have
$$
{}^t\ol{h}1_nh = \eta_\pi^{-1}{}^th_2 1_n h_1\eta_\pi^{-1} = \eta_\pi 1_n \eta_\pi = T_1,
$$ 
which is impossible for $h \in GL_n(k')$. 
\end{proof}

\bigskip
We consider
$$
\N = \set{(\twomatrix{h_1}{0}{0}{k_1}, \twomatrix{h_2}{0}{0}{k_2})}{(h_1, h_2), \; (k_1, k_2) \in \U(1_n)}
$$
and the similar $\Gamma$-action on $\N$ as before. Then we may identify $\N^\Gamma$ with 
$$
N = \set{\twomatrix{h}{0}{0}{k}}{h, k \in U(1_n)}.
$$
Setting
$$
y_0 = \twovector{\xi1_n}{1_n} \in \frX_{1_n}, \quad \xi = \frac{1 + \sqrt{\ve}}{2},
$$
we see the stabilizer in $U(H_n)$ at $y_0 U(1_n) \in X_{1_n}$  is given by  (cf.~\cite[Lemma~1.1]{Oda}) 
$$
\nu N \nu^{-1}, \qquad \nu = \twomatrix{\xi 1_n}{\ol{\xi}1_n}{1_n}{-1_n} \in GL_{2n}(\calO_{k'}).
$$
On the other hand, as is noted in (\ref{stabilizer at 1_2n}) 
the stabilizer in $G = U(j_{2n})$ at $1_{2n}$ is given by
$$
\mu N \mu^{-1}, \qquad \mu = \twomatrix{1_n}{1_n}{j_n}{-j_n} \in GL_{2n}(k).  
$$
Since
\begin{eqnarray*}
&&
H_n[\nu] = \twomatrix{\ol{\xi}1_n}{1_n}{\xi 1_n}{-1_n} \twomatrix{1_n}{-1_n}{\xi 1_n}{\ol{\xi}1_n} 
= \twomatrix{1_n}{0}{0}{-1_n},\\
&&
j_{2n}[\mu] = \twomatrix{1_n}{j_n}{1_n}{-j_n} \twomatrix{1_n}{-1_n}{j_n}{j_n} = 2 \twomatrix{1_n}{0}{0}{-1_n},
\end{eqnarray*}
we have
\begin{eqnarray} \label{isom}
\mu\nu^{-1} U(H_n) \nu \mu^{-1} = G, \quad
\mu\nu^{-1} \U(H_n) \nu \mu^{-1} = G(\ol{k}). 
\end{eqnarray}
where we identify $\nu$ and $\mu$ with their images in $R_{k'/k}(GL_{2n})$.
Thus we have 
\begin{equation} \label{diag 2}
\begin{array}{ccccccccc}
\X_{1_n}(\ol{k}) &\cong & {} & \U(H_n) \big{/} \nu \N \nu^{-1} & \stackrel{\vphi}{\stackrel{\sim}{\longrightarrow}}& G(\ol{k}) \big{/} \mu \N \mu^{-1} &\cong& {}& X_n(\ol{k}) \\[2mm]
\cup & {}&{} & {}&{}& {} & {}&{} & \cup \\[2mm]
\X_{1_n}(k) &\supset X_{1_n} &\cong& U(H_n) \big{/} \nu N \nu^{-1} & \stackrel{\vphi}{\stackrel{\sim}{\longrightarrow}} & G \big{/} \mu N \mu^{-1} \cong & G \cdot 1_{2n} & \subset & X_n ,\\ 
\end{array}
\end{equation}
where $\vphi$ is the conjugation determined by (\ref{isom}). Then we have the following by the commutative diagram (\ref{diag 2}), Lemma~\ref{Hn-orbit}, and Theorem~\ref{thm: G-orbits}.

\medskip
\begin{prop}
The above $\vphi$ gives an isomorphism between the sets of $k$-rational points 
$$
U(H_n) \backslash \X_{1_n}(k) \cong G \backslash X_n,
$$
and $U(H_n)$-orbit decomposition
$$
\X_{1_n}(k) = X_{1_n} \sqcup f(X_{T_1}); \quad X_{1_n} \cong G \cdot x_0, \quad 
f(X_{T_1}) \cong G \cdot x_1, 
$$
where $x_0$ and $x_1$ are the representatives of $G$-orbits in $X_n$ given in Theorem~\ref{thm: G-orbits}.
\end{prop}

\medskip
Now we assume that $q$ is odd and consider the Cartan decomposition of $X_T = \frX_T/U(T)$. 
Set $K' = U(H_n) \cap GL_{2n}(\calO_{k'})$ and recall $K = G \cap GL_{2n}(\calO_{k'})$. Since $2 \notin (\pi)$, we see $\mu \in GL_{2n}(\calO_{k'})$ and 
$$
\mu \nu^{-1} K' \nu \mu^{-1} = K.
$$
Hence we have the bijective correspondence 
$$
K' \backslash \X_{1_n}(k) \longleftrightarrow K \backslash X_n, 
$$
and we see the space $X_T$ inherits the Cartan decomposition of $X_n$, the space of unitary hermitian space.  Thus we have the following.

\medskip
\begin{thm} \label{thm: oda}
Assume $k$ has odd residual characteristic and take any $T \in \calH_n(k')$. Then
\begin{eqnarray*}
\frX_T = \dsqcup{\shita{\lam \in \Lam^+}{\lam \sim T}}\, 
K' x_\lam h_\lam U(T), 
\end{eqnarray*}
where 
$$
x_\lam = \twovector{\xi \pi^\lam}{1_n} \in \frX_{\pi^\lam},  
$$
$\lam \sim T$ means that $\abs{\lam} \equiv v_\pi(\det(T)) \pmod{2}$ and guarantees the existence of $h_\lam \in GL_n(k')$ satisfying $T = \pi^\lam[h_\lam]$.
\end{thm}

The above decomposition has been expected in \cite[Remark~4.2]{Oda}. 
In \cite{Oda}, where we have known the disjointness of orbits in the right hand side by explicit formulas of spherical functions $\omega_T(y; z)$, but we didn't know they are enough. By Theorem~\ref{thm: oda}, we see the spherical Fourier transform $F_T$ is isomorphic in \cite[Theorem~4.1]{Oda} if $q$ is odd.

\bibliographystyle{amsalpha}

\end{document}